\documentclass[11pt]{amsart}
 
\usepackage[T1]{fontenc}
\usepackage[utf8]{inputenc}
\usepackage{amsmath}
\usepackage{amsfonts}
\usepackage{amssymb}
\usepackage{amsthm}
\usepackage{yhmath}

\usepackage{hyperref}
\usepackage{enumitem}
\usepackage{graphicx}

\numberwithin{equation}{section}

\newtheorem{theorem}[equation]{Theorem}
\newtheorem{proposition}[equation]{Proposition}
\newtheorem{lemma}[equation]{Lemma}
\newtheorem{corollary}[equation]{Corollary}
\newtheorem{conjecture}[equation]{Conjecture}

\makeatletter
\g@addto@macro\th@definition{\thm@headpunct{}}
\makeatother
\theoremstyle{definition}
\newtheorem{remark}[equation]{Remark}
\newtheorem*{remarks*}{Remarks}
\newtheorem{definition}[equation]{Definition}

\newcommand{\abs}[1]{\left\vert#1\right\vert}
\newcommand{\norm}[1]{\left\Vert#1\right\Vert}

\newcommand{\vol}{\operatorname{vol}}

\newcommand{\subriemannian}{sub-Riemannian } 

\newcommand{\subfinsler}{sub-Finsler } 

\newcommand{\R}{\mathbb R}
\newcommand{\N}{\mathbb N}
\newcommand{\PP}{\mathcal P}
\newcommand{\g}{\mathfrak{g}}
\newcommand{\Span}{\operatorname{span}}
\def\eps{\epsilon}
\newcommand{\Ad}{\operatorname{Ad}}
\newcommand{\ad}{\operatorname{ad}}

\newcommand{\Width}{\operatorname{MinHeight}} 
\newcommand{\Size}{\operatorname{Size}} 
\newcommand{\matrixcomplement}[3]{#1_{(#2,#3)}}
\newcommand{\brkt}[2]{\left[#1,#2\right]}
\newcommand{\dil}[1]{\delta_{#1}} 
\newcommand{\proj}{\pi} 
\newcommand{\stepproj}{\proj_{s-1}} 
\newcommand{\multistepproj}[1]{\proj_{#1}} 
\newcommand{\horproj}{\proj} 
\newcommand{\idelem}{1_G} 
\newcommand{\idengel}{1_\engel} 
\newcommand{\Tang}{\mathrm{Tang}} 
\newcommand{\WeakTang}{\mathrm{WeakTang}} 
\newcommand{\Asymp}{\mathrm{Asymp}} 
\newcommand{\image}{\mathrm{Im}} 
\newcommand{\matelem}[3]{#1_{#2#3}} 
\newcommand{\configuration}{E} 
\newcommand{\innerproduct}[2]{\left\langle #1,#2\right\rangle} 
\newcommand{\engel}{E}
\newcommand{\dhaus}{d_H}

\begin{document}
\title[Blowups and blowdowns of Carnot geodesics]{Blowups and blowdowns of geodesics \\in Carnot groups}

\author[Eero Hakavuori]{Eero Hakavuori} 
\email{eero.hakavuori@sissa.it}

\author[Enrico Le Donne]{Enrico Le Donne}
\email{enrico.ledonne@unifr.ch}

\address[Hakavuori]{SISSA, Via Bonomea 265, 34136 Trieste, Italy}
 
\address[Le Donne]{Department of Mathematics, University of Fribourg, Chemin du Musée~23, 1700 Fribourg, Switzerland \& 
University of Jyv\"askyl\"a, Department of Mathematics and Statistics, P.O. Box (MaD), FI-40014, Finland}


\begin{abstract}
This paper provides some partial regularity results for geodesics (i.e., isometric images of intervals) in arbitrary \subriemannian and \subfinsler manifolds.
Our strategy is to study infinitesimal and asymptotic properties of geodesics in Carnot groups equipped with arbitrary \subfinsler metrics. We show that tangents of Carnot geodesics are geodesics in some groups of lower nilpotency step. Namely, every blowup curve of every geodesic in every Carnot group is still a geodesic in the group modulo its last layer. 
Then as a consequence we get that in every \subriemannian manifold any $s$ times iterated tangent of any geodesic is a line, where $s$ is the step of the \subriemannian manifold in question.
With a similar approach, we also show that blowdown curves of geodesics in \subriemannian Carnot groups are contained in subgroups of lower rank. This latter result is also extended to rough geodesics. 
\end{abstract}


\thanks{E.H. was supported by the Vilho, Yrj\"o and Kalle V\"ais\"al\"a Foundation.
E.L.D. was partially supported by the Academy of Finland 
(grant 288501 `\emph{Geometry of subRiemannian groups}')
and by the European Research Council 
(ERC Starting Grant 713998 GeoMeG `\emph{Geometry of Metric Groups}').
}
\date{January 17, 2022}
\maketitle
\tableofcontents
 
\newpage
\section{Introduction}
{
In \subriemannian geometry, one of the major open problems is the regularity of geodesics, i.e., of isometric embeddings of intervals. Because of the presence of abnormal curves, a priori \subriemannian geodesics only have Lipschitz regularity, yet all known examples are $C^\infty$. For a modern introduction to the topic we refer to \cite{Vittone-note}.

We approach the differentiability problem by considering the infinitesimal geometry, which is given by \subriemannian Carnot groups, and within them studying limits of dilated curves, called tangents or blowups. The main aim of this paper is to show that iterating the process of taking tangents one necessarily obtains only lines:

\begin{theorem}\label{thm:sr tangent iteration}
	If $\gamma$ is a geodesic in a \subriemannian manifold, then every $s$ times iterated tangent of $\gamma$ is a line, where $s$ is the step of the \subriemannian manifold.
\end{theorem}

This is a generalization of other partial results that have already been attained using a similar approach: In \cite{Hakavuori_LeDonne_2016} we showed that tangents of geodesics are not corners, and in \cite{Monti_Pigati_Vittone} it is shown that among all tangents at a point, one of the tangents is a line. Here ``line'' means a left translation of a one-parameter subgroup and ``corner'' means two half-lines joined together not forming a line.
 
A basic fact from metric geometry is that tangents of geodesics are themselves infinite geodesics. Therefore knowledge about infinite geodesics can help understand the regularity problem. For this reason, in this present work in addition to tangents we consider asymptotic cones, also called blowdowns, of infinite geodesics in Carnot groups. 

Before stating our more specific results, we shall specify the notion of tangents. The notion is the same as previously used in \cite{Hakavuori_LeDonne_2016, Monti_Pigati_Vittone}. 
We shall mainly restrict our considerations to Carnot groups, while allowing arbitrary length distances. For the notion of tangents within manifolds, we refer to \cite{Monti_Pigati_Vittone:tangent_cones}.

Let $G$ be a \subfinsler Carnot group, cf. the standard definition in \cite{LeDonne:Carnot}. In $G$ we have a Carnot-Carath\'eodory distance $d$ defined by a norm on the horizontal space $V_1$ of $G$, and we have a one-parameter family of dilations, denoted by $(\delta_h)_{h>0}$. 
Let $I$ be an open interval in $\R$, possibly $I=\R$. 
Let $\gamma:I\to G$ be a 1-Lipschitz curve and fix $\bar t \in I$. 
Denote by $\gamma_h: I_h \to G$ the curve defined on $I_h := \tfrac{1}{h} (I - \bar t)$ by
\begin{equation*}
\gamma_h(t) := \delta_{\frac{1}{h}} \left( \gamma(\bar t)^{-1} \gamma( \bar t + ht )\right).
\end{equation*}
Notice that the last definition is just the non-abelian version of the difference quotient used in the definition of derivatives. 
It is trivial to check that $\gamma_h$ is 1-Lipschitz and $\gamma_h(0)=\idelem$ for all $h\in (0,\infty)$. 
Consequently, by Ascoli-Arzel\'a, for every sequence $h_j\to 0$ there is a subsequence $h_{j_k}$ and a curve $\sigma:\R \to G$ such that $\gamma_{h_{j_k}}\to \sigma$ uniformly on compact sets of $\R$. 
Hence, we define the collection of {\em tangents} as the nonempty set
\[ \Tang(\gamma,\bar t) := \left\{ \sigma\; \vert\; \exists h_j\to 0 :\gamma_{h_{j }}\to \sigma \right\}. \]
\begin{remark}\label{rmk:unique tangent}
	For any Lipschitz curve $\gamma: I\to G$, the tangent cone $\Tang(\gamma,\bar t)$ is a singleton if and only if $\gamma$ has both one-sided derivatives at $\bar{t}$. 
	To see the forward implication, we observe that the unique tangent must be dilation invariant and hence either a line or a corner, so the one-sided derivatives exist.
	For the reverse implication, we first observe that if a finite length curve has a one-sided derivative at a point, then the derivative can only be a horizontal vector. It then follows immediately from the definition of the tangent cone that there is a unique tangent curve, which is the concatenation of the two half-lines defined by the one-sided derivatives.
\end{remark}

In the case where $\gamma:I\to M$ is a Lipschitz curve on a \subriemannian or \subfinsler manifold $M$, we will also denote by $\Tang(\gamma,\bar{t})$ the collection of metric tangents of $\gamma$ at $\bar{t}$. In this case, the elements of $\Tang(\gamma,\bar{t})$ are no longer curves in $M$, but instead curves in the metric tangent space $\tilde{M}$, also called the nilpotent approximation of $M$. We refer to \cite[Section 2.3.1]{jeancontrol} and \cite{Monti_Pigati_Vittone:tangent_cones} for details on this more general construction.

When $I=\R$, we will also consider limits of curves $\gamma_{h_j}$ for sequences $h_j\to\infty$. Similarly to the case $h_j\to 0$, for every sequence $h_j\to\infty$ there is a subsequence $h_{j_k}$ and a curve $\sigma:\R \to G$ such that $\gamma_{h_{j_k}}\to \sigma$ uniformly on compact sets of $\R$, so we define the collection of {\em asymptotic cones} as the nonempty set 
\[ \Asymp(\gamma) := \left\{ \sigma\; \vert\; \exists h_j\to \infty :\gamma_{h_{j }}\to \sigma \right\} . \]
The definition of $\Asymp(\gamma) $ is independent on the choice of $\bar t$ and technically the assumption that $I=\R$ is not necessary if we use the domains $I_h$ as in the definition of tangents. However if $I$ is bounded, the domains $I_h$ degenerate to a point, and in the case where $I$ 
is a half-line, all arguments are only superficially different from the line case.

Finally, we define the iterated tangent cones as the set of all tangents of (iterated) tangents, i.e., for each $k\geq 1$ we define
\[ \Tang^{k+1}(\gamma,\bar t) := \bigcup_{\sigma\in\Tang^k(\gamma,\bar t)} \Tang(\sigma,0). \]
The elements $\sigma\in\Tang^{k}(\gamma,\bar t)$ for any $\bar{t}$ are called \emph{$k$ times iterated tangents of $\gamma$}. We remark that a simple diagonal argument\footnote{If $\gamma_{h_j}\to \sigma$ and $\sigma_{k_j}\to \eta$ for some ${h_j}, {k_j}\to 0$, then for all $\ell$ we have $\gamma_{k_\ell h_j }\to \sigma_{k_\ell}$ and so, by a diagonal argument, there is a sequence $\ell_j$ such that $\gamma_{k_{\ell_j} h_j }\to \eta$.} shows that iterated tangents are also tangents, i.e., that
\[ \dots\subset\Tang^{k+1}(\gamma,\bar{t})\subset \Tang^k(\gamma,\bar{t})\subset\dots\subset \Tang(\gamma,\bar{t}).  \]

Assume $\gamma:I\to G$ is a geodesic, i.e., $d(\gamma(a), \gamma(b) ) = |a-b|$, for all $a,b\in I$.
Our main results in the Carnot group setting are that every element in $\Tang(\gamma,\bar t) $ is a geodesic also when projected into some quotient group of lower step, and that every element in $\Asymp(\gamma)$ is a geodesic inside some subgroup of lower rank (see Theorem~\ref{thm:tangent} and Corollary~\ref{cor:subriemannian blowdowns}, respectively).

\subsection{Statement of the results}
Unless otherwise stated, in what follows $G$ will be a \subfinsler Carnot group of nilpotency step $s$ and $V_1\oplus\dots\oplus V_s=\g$ will be the stratification of the Lie algebra $\g$ of $G$. We denote by $\horproj:G\to G/\brkt{G}{G}$ the projection on the abelianization and by $\stepproj :G\to G/\exp(V_s)$ the projection modulo the last layer $V_s$ of $\g$. 

Both groups $G/\brkt{G}{G}$ and $ G/\exp(V_s)$ are canonically equipped with structures of \subfinsler Carnot groups (see Proposition~\ref{prop:quotient}).
The normed vector space $G/\brkt{G}{G}$ is also further canonically identified with the first layer $V_1$ and its dimension is the rank of $G$. The group $G/\exp(V_s)$ has nilpotency step $s-1$, one lower than the original group $G$.
\begin{theorem}[Blowup of geodesics]\label{thm:tangent}
	If $\gamma:I\to G$ is a geodesic and $t\in I$,
	then for every $\sigma\in\Tang(\gamma,t)$, the curve $\stepproj\circ\sigma:\R\to G/\exp(V_s)$ is a geodesic.
\end{theorem}

This result implies the previously known ones from \cite{Hakavuori_LeDonne_2016} that corners are not minimizing and from \cite{Monti_Pigati_Vittone} that in the \subriemannian case one of the tangents is a line. In fact, iterating the above result, we get the following corollary.

\begin{corollary}\label{cor:tangent iteration}
	If $\gamma:I\to G$ is a geodesic and $t\in I$,
	then for every $\sigma\in\Tang^{s-1}(\gamma,t)$, the horizontal projection $\horproj\circ\sigma$ is a geodesic.
	In particular, if $G$ is \subriemannian then every $\sigma\in\Tang^{s-1}(\gamma,t)$ is a line.
\end{corollary}

The results of Theorem~\ref{thm:tangent} and Corollary~\ref{cor:tangent iteration} also hold for a more general notion of tangents, usually called \emph{weak tangents}, see Theorem~\ref{thm:weak tangent} and Corollary~\ref{cor:weak tangent iteration}.
In the \subriemannian setting, since all infinite geodesics in step 2 are lines, Theorem~\ref{thm:sr tangent iteration} and Corollary~\ref{cor:tangent iteration} can be improved slightly, decreasing the number of iterations needed from $s$ and $s-1$ to $s-1$ and $s-2$, respectively.

As applications of the existence of a line tangent, we show that in every non-Abelian Carnot group where in the abelianization the  infinite geodesics are lines, there is always a geodesic that loses optimality whenever it is extended (see Proposition~\ref{Prop:lost:min}), and show that the non-minimality of corners holds also in the non-constant rank case (see Proposition~\ref{prop:non-constant rank corners}).

As mentioned in the introduction, every element in $\Tang(\gamma,t)$ is an infinite geodesic. We provide next other results that are valid for any infinite geodesics regardless of whether or not they are tangents.

\begin{theorem}\label{thm:geodesic blowdowns} If
	$\gamma:\mathbb{R}\to G$ is a geodesic such that $\horproj\circ\gamma:\mathbb{R}\to G/\brkt{G}{G}$ is not a geodesic, then there exist $R>0$ and a hyperplane $W\subset V_1$ such that $\image(\horproj\circ\gamma)\subset B_{V_1}(W,R)$.
\end{theorem}
In the above theorem, we denote by $ B_{V_1}(W,R)$ the $R$-neighborhood of $W$ within $V_1$. To prove Theorem~\ref{thm:geodesic blowdowns} we shall adopt a wider viewpoint. In fact, we will consider rough geodesics and still have the same rigidity result (see Theorem~\ref{thm:quasigeod}). 

It is possible that the claim of Theorem~\ref{thm:geodesic blowdowns} could be strengthened to say that the projection of the geodesic is asymptotic to the hyperplane. In Corollary~\ref{cor:Engel geodesic asymptotic properties}, we show that this is true for the only known family of examples of non-line infinite geodesics, arising from the explicit study of geodesics in the Engel group, see \cite{Ardentov_Sachkov:engel_cut_time}. We also show that each of these geodesics is in a finite neighborhood of a line in the Engel group itself. However, by lifting the same geodesics to a step 4 Carnot group, we show that there exist infinite geodesics that are not in a finite neighborhood of any line (see Corollary~\ref{cor:geodesic not in a line nbhd}).

Since in Euclidean spaces the only infinite geodesics are the straight lines, an immediate consequence of Theorem~\ref{thm:geodesic blowdowns} is the following.

\begin{corollary}[Blowdown of geodesics]\label{cor:subriemannian blowdowns} If $\gamma$ is a geodesic in a \subriemannian Carnot group $G\neq \R$, then there exists a proper Carnot subgroup $H<G$ containing every element of $\Asymp(\gamma)$.
\end{corollary}
As with Theorem~\ref{thm:geodesic blowdowns}, Corollary~\ref{cor:subriemannian blowdowns} admits a generalization for rough geodesics (see Corollary~\ref{cor:subriemannian quasiblowdowns}). As a stepping stone to this generalization, we also prove that rough geodesics in Euclidean spaces have unique blowdowns (see Lemma~\ref{Euclid:rough:geodesic}).

Similarly as with Theorem~\ref{thm:tangent}, we can iterate Corollary~\ref{cor:subriemannian blowdowns} and deduce that some blowdown of an infinite geodesic in a \subriemannian Carnot group must be a line.
Furthermore, we show that in \subriemannian Carnot groups, every blowdown of an infinite geodesic is a line or an abnormal geodesic (see Proposition~\ref{prop:abnormal blowdowns}).

\subsection{Organization of the paper}
In Section~\ref{Sec2} we discuss technical lemmas based on linear algebra and our error correction procedure. We introduce the concepts of minimal height and size. Proposition~\ref{prop:lower step triangle inequality} 
is the crucial  estimate and is a variant of a triangle inequality with an error term depending on the notion of size. 
This proposition is the key ingredient for both the proof of Theorem~\ref{thm:tangent} and the proof of
Theorem~\ref{thm:geodesic blowdowns}. 

In Sections \ref{sec:UP} and \ref{sec:DOWN} we prove our main results. Section~\ref{sec:UP} covers our results about tangents of geodesics: Theorems \ref{thm:sr tangent iteration} and \ref{thm:tangent}, and Corollary~\ref{cor:tangent iteration}, and their stronger forms for weak tangents in Theorem~\ref{thm:weak tangent} and Corollary~\ref{cor:weak tangent iteration}.
We also give a quantified version in Theorem~\ref{thm:quantified tangent}, which 
expresses the  extent to which 
the projection of a geodesic may fail to be minimizing.
 Section~\ref{sec:DOWN} covers our results about infinite geodesics and blowdowns: Theorem~\ref{thm:geodesic blowdowns} and Corollary~\ref{cor:subriemannian blowdowns}, and their rough counterparts: Theorem~\ref{thm:quasigeod} and Corollary~\ref{cor:subriemannian quasiblowdowns}.

In Sections \ref{sec:infinite geodesics} and \ref{sec:exists line tangent} we discuss some applications of our main results. In Section~\ref{sec:infinite geodesics} we consider infinite geodesics, and prove the statement about abnormality of blowdowns (Proposition~\ref{prop:abnormal blowdowns}).
In Section~\ref{sec:exists line tangent} we consider applications of the existence of a line tangent. We prove the existence of non-extendable geodesics in non-abelian Carnot groups (Proposition~\ref{Prop:lost:min}) and the non-minimality of corners in non-constant rank \subriemannian manifolds (Proposition~\ref{prop:non-constant rank corners}).

In Section~\ref{sec:Sharp} we discuss to which extent one can expect an improvement of the blowdown result Theorem~\ref{thm:geodesic blowdowns}, restricting our attention to rank-2 Carnot groups. 
In Section~\ref{sec:Lines} we cover preliminaries on lines in Carnot groups and study when two lines are at bounded distance. In Section~\ref{sec:Engel} we consider the example of an infinite non-line geodesic in the Engel group. We use this curve to find a counter-example to one possible strengthening of Theorem~\ref{thm:geodesic blowdowns}.

\section{Preliminaries: minimal height, size, and error correction} \label{Sec2}

\subsection{Minimal height of a parallelotope and its properties}
\begin{definition}[Minimal height of a parallelotope]\label{defn:width}
Let $V$ be a normed vector space with distance $d_V$. The \emph{minimal height} of an $m$-tuple of points $(a_1,\dots,a_m)\in V^m $ is the smallest height of the parallelotope generated by the points, i.e., 
\[ 
\Width(a_1,\dots,a_m) = \min_{j\in\{1,\dots,m\}} d_V(a_j,\Span\{a_1,\dots,\hat{a}_j,\dots,a_m\}).\]
\end{definition}

\stepcounter{equation}
\begin{remarks*}
	\begin{enumerate}[wide, labelwidth=!, labelindent=0pt,label={\textbf{\arabic{section}.\arabic{equation}.\arabic*}},ref={\arabic{section}.\arabic{equation}.\arabic*}]
		\item\label{rmk:width:linear independence and width} Points $a_1,\dots,a_m$ in a normed vector space are linearly independent if and only if $\Width(a_1,\dots,a_m)\neq0$.
		\item\label{rmk:width:volume and width} Assume $V$ is a Euclidean space $\R^r$ and denote by $\vol_m$ the usual $m$-dimensional volume.
		Let $\PP(a_1,\dots,a_m)$ denote the parallelotope generated by the vectors $a_1,\dots,a_m$.
		Notice that the volume of $\PP( a_1,\dots,a_m )$ equals the volume of any base $\PP(a_1,\dots,\hat{a}_j,\dots,a_m)$
		times the corresponding height, which is $d(a_j,\Span\{a_1,\dots,\hat{a}_j,\dots,a_m\})$.
		Hence, we have
		\begin{eqnarray*}
			\Width(a_1,\dots,a_m) &=& \min_{j\in\{1,\dots,m\}} \frac{\vol_m \PP( a_1,\dots,a_m ) } { \vol_{m-1} \PP(a_1,\dots,\hat{a}_j,\dots,a_m) }\\
			&=& \frac{\vol_m \PP( a_1,\dots,a_m ) } { \max_{j\in\{1,\dots,m\}} \vol_{m-1} \PP(a_1,\dots,\hat{a}_j,\dots,a_m) }.
		\end{eqnarray*}
		Hence, if $\PP^*:=\PP(a_1,\dots,\hat{a}_j,\dots,a_m) $ is a face of the parallelotope with maximal $(m-1)$-dimensional volume, then
		\[ \Width(a_1,\dots,a_m) = \frac{\vol_m \PP( a_1,\dots,a_m ) } { \vol_{m-1} \PP^* } = d( a_j, \Span\PP^*). \]
	\end{enumerate}
\end{remarks*}
We next prove a basic lemma that uses the notion of minimal height to bound the entries of the inverse of a matrix. This bound will then be used in Lemma \ref{lemma:correcting system}.
\begin{lemma}\label{lemma:Cramer:width}
Let $A$ be a matrix with columns $A_1,\ldots,A_r\in\R^r$.
If $\Width(A_1,\dots,A_r)>0$, then $A$ is invertible and its inverse $B$ has entries $\matelem{B}{k}{j}$ bounded by
\[ \abs{ \matelem{B}{k}{j} }\leq \frac{1}{ \Width(A_1,\dots,A_r) } , \qquad \forall k,j=1,\dots,r. \]
\end{lemma}
\begin{proof}
The fact that $A$ is invertible follows from Remark~\ref{rmk:width:linear independence and width}. For the estimate on the entries of the inverse, we will use a well-known formula from linear algebra (see \cite[page 219]{Lang:linear_algebra_2}):
If $\matrixcomplement{A}{k}{j}$ denotes the matrix $A$ with row $k$ and column $j$ removed, then the entries of $B$ can be calculated by 
\begin{equation}\label{Cramer} 
\matelem{B}{k}{j} = (-1)^{k+j}\frac{\det \matrixcomplement{A}{k}{j}}{\det A}.
\end{equation}
Fix $j,k\in\{1,\dots,r\}$.
Let $P_k:\R^r\to \R^{r-1}$ be the projection that forgets the $k$-th coordinate:
	\[ P_k(y_1,\dots,y_r) := (y_1,\dots,\hat{y}_k,\dots,y_r). \]
Consider the following parallelotopes:	
Let $\PP$ be the $r$-parallelotope in $\R^r$ determined by the points $A_1,\dots,A_r$, let $\PP_j$ be the $(r-1)$-parallelotope in $\R^r$ determined by the same points excluding the vertex $A_j$, and let $\PP_j^k=P_k( \PP_j )$, which is an $(r-1)$-parallelotope in $\R^{r-1}$. 

The geometric interpretation of the determinant states that 
\[ \abs{\det A} = \vol_r(\PP)\quad \text{ and }\quad \abs{\det \matrixcomplement{A}{k}{j}} = \vol_{r-1}(\PP_j^k). \]
Moreover, since $P_k( \PP_j )= \PP_j^k$ and the projection $P_k$ is $1$-Lipschitz, we have
\[ \vol_{r-1}(\PP_j^k)\leq\vol_{r-1}( \PP_j ). \]
By these last two observations, we have that
\begin{equation}\label{eq:det quotient as volumes}
	\frac{\abs{\det \matrixcomplement{A}{k}{j}}}{\abs{\det A}} = \frac{\vol_{r-1}(\PP_j^k)}{\vol_r(\PP)} \leq \frac{\vol_{r-1}( \PP_j )}{\vol_r(\PP)}.
\end{equation}
Let $L_j:=\Span\{a_1,\dots,\hat{a}_j,\dots,a_m\}$.
Since $L_j$ is the span of $\PP_j$ and $\PP_j$ is a face of $\PP$, we calculate the volume of $\PP$ as in Remark~\ref{rmk:width:volume and width} as
\begin{equation}\label{eq:volume as height times area}	 
\vol_r(\PP) = d(a_j,L_j)\vol_{r-1}(\PP_j).
\end{equation}
By the definition of $\Width$ as the minimum of the distances $d(a_j,L_j)$, we conclude that
\[ \abs{\matelem{B}{k}{j}} 
\overset{\eqref{Cramer}}{=} \frac{\abs{\det\matrixcomplement{A}{k}{j}}}{\abs{\det A}}
\overset{\eqref{eq:det quotient as volumes}}{\leq} 
\frac{\vol_{r-1}(\PP_j)}{\vol_r(\PP)}
\overset{\eqref{eq:volume as height times area}}{=}
\frac{1}{d(a_j,L_j)}
\leq\frac{1}{\Width(A_1,\dots,A_r)}. \qedhere \]
\end{proof}

\begin{remark}\label{rmk:singular values}
	For our purposes, the notion of minimal height $$\Width(a_1,\ldots,a_r)$$ could also be replaced by the smallest singular value $$s_{\min}(A) = \min_{\norm{x}=1}\norm{Ax}$$ of the matrix $A=\begin{bmatrix}a_1 & \cdots & a_r\end{bmatrix}$. The relevant use of the minimal height is through the bound of Lemma~\ref{lemma:Cramer:width} for the entries of the inverse matrix of $A$. Since the operator norm $\norm{A^{-1}}$ is bounded by $1/s_{\min}(A)$, we would obtain a similar estimate as in Lemma~\ref{lemma:Cramer:width} up to a constant factor using the smallest singular value.
\end{remark}

\subsection{Size of a configuration and error correction}
\begin{definition}\label{def:submetry}
	A map $\pi: X\to Y$ between metric spaces is a \emph{submetry} if $\pi(B(x,r)) = B(\pi(x),r)$ for all $x\in X$ and $r>0$.
\end{definition}

\begin{proposition}\label{prop:quotient}
	On $G/\brkt{G}{G}$ and on $G/\exp(V_s)$ there are canonical structures of \subfinsler Carnot groups such that the projections $\horproj:G\to G/\brkt{G}{G}$ and $\stepproj:G\to G/\exp(V_s)$ are submetries. 
	In particular, for any $g_1,g_2\in G$ there exists $h\in\exp(V_s)$ such that
	\[ d(\stepproj(g_1),\stepproj(g_2)) = d(g_1,hg_2). \]
\end{proposition}
\begin{proof}
	This proof is well known. It probably goes back to Berestovskii \cite[Theorem 1]{b1}. The key point here is that both $\exp(V_s)$ and $\brkt{G}{G}$ are normal subgroups $N\triangleleft G$. Thus one can define the distance between two points $p,q\in G/N$ in the quotient as the distance 
	\begin{equation*}
	d_{G/N}(p,q) = d(\pi^{-1}(p),\pi^{-1}(q))
	\end{equation*}
	between their preimages under the projection $\pi:G\to G/N$. 
	The reader can find the details in \cite[Corollary 2.11]{LeDonne_Rigot_rmknobcp}.
	
	With this definition, the resulting metric spaces are necessarily \subfinsler Carnot groups by the characterization of \cite{LeDonne_characterization}. Moreover for all $g_1,g_2\in G$ we have $d_{G/N}(\pi(g_1),\pi(g_2))\leq d(g_1,g_2)$ and by local compactness there exists a point $g_2'\in\pi^{-1}(\pi(g_2))$ such that
	\begin{equation*}
	d_{G/N}(\pi(g_1),\pi(g_2)) = d(g_1,g_2').
	\end{equation*}
	Since $g_2$ and $g_2'$ are in the same coset, we find that $g_2' = hg_2$ for some $h\in N$.
\end{proof}

Note that $G/[G,G]$ is abelian, so the canonical \subfinsler structure on $G/[G,G]$ is the structure of a normed vector space.

\begin{definition}[Size of a configuration]
	Let $G$ be a \subfinsler Carnot group. The \emph{size} of an $(m+1)$-tuple of points $(g_0,\dots,g_m)\in G^{m+1}$ is
	\begin{equation}\label{eqdef:size}
		\Size(g_0,\dots,g_m) = \Width (\horproj(g_1)-\horproj(g_0),\horproj(g_2)-\horproj(g_1),\dots,\horproj(g_m)-\horproj(g_{m-1})).
	\end{equation}
\end{definition}

\begin{remark}\label{rmk:width:general position and size}
	Remark \ref{rmk:width:linear independence and width} states that non-zero $\Width$ characterizes linear independence of points. Analogously, $\Size(g_0,\dots,g_m)\neq 0$ if and only if the horizontal projections $\horproj(g_0),\dots,\horproj(g_m)\in G/\brkt{G}{G}$ are in general position.
\end{remark}

The reason to consider this notion of size stems from Lemma~\ref{lemma:perturbation lemma} below, which describes our error correction procedure. Within this lemma, we need to bound the norms of solutions to a certain linear system. A convenient bound is given in Lemma~\ref{lemma:correcting system} in terms of the size of a configuration of points. This dependence of the bound of solutions on the size of a configuration is the reason we are able to give restrictions on the behavior of tangents and asymptotic cones of geodesics.

\begin{lemma}[Linear system of corrections]\label{lemma:correcting system}
	For every Carnot group $G$ of rank $r$ and step $s\geq2$, there exists a constant $K>0$ with the following property:
	
	Let $x_0, \ldots, x_r\in G$ and $X_j:= \log( x_{j-1}^{-1} x_j)$, for $j=1,\ldots, r$. 
	If $\Size(x_0, \ldots, x_r) >0$, then for every $Z \in V_{s}$ there exist $Y_1,\dots,Y_{r}\in V_{s-1}$ such that
	\begin{equation}\label{lemma:correcting1}
	\brkt{Y_1}{X_1}+\dots+\brkt{Y_{r}}{X_r}=Z 
	\end{equation}
	and 
	\begin{equation}\label{lemma:correcting2}
	d(\idelem,{\exp(Y_j)})^{s-1}\leq K
	\frac{d(\idelem,{\exp(Z)})^{{s}}}{\Size(x_0, \ldots, x_r)}, \qquad \forall j\in 1,\ldots, r.
	\end{equation}
\end{lemma}
\begin{proof}
Fix arbitrary norms on the vector spaces $V_{s-1}$ and $V_{s}$, and denote them generically as $\norm{\cdot}$.
Observe that the functions
\[ W\mapsto d(\idelem,\exp(W) )^{s-1} \quad\text{and}\quad
{Z}\mapsto d(\idelem,\exp(Z))^{s} \]
are 1-homogeneous with respect to scalar multiplication.
Therefore there exists a constant $C_1>1$ such that 
\begin{equation}\label{normCuno}
\norm{W}\simeq_{C_1} d(\idelem,\exp(W ) )^{s-1} \quad\text{and}\quad
\norm{Z}\simeq_{C_1} d(\idelem,\exp(Z))^{s},
\end{equation} 
where $a \simeq_{c} b $ stands for $b/c\leq a \leq cb$. 

Fix a basis $\bar X_1,\dots,\bar X_{r}$ of $V_1$.
Observe that the map $(W_1,\ldots,W_r)
\mapsto\brkt{W_1}{\bar X_1}+\dots+\brkt{W_r}{\bar X_r}$ is a linear surjection between the normed vector spaces $(V_{s-1} )^r$ and $V_{s}$, where on $(V_{s-1})^r$ we use the norm $ \max_{ i=1,\ldots,r }\{\norm{W_i}\}$.
Thus the map can be restricted to some subspace so that it becomes a biLipschitz linear isomorphism.
In other words, there exists a constant $C_2>1$ such that for all 
$Z \in V_{s}$ there exist vectors $W_1,\dots,W_{r}\in V_{s-1}$ such that
\begin{equation}\label{solution:W}
Z=\brkt{W_1}{\bar X_1}+\dots+\brkt{W_r}{\bar X_r}
\end{equation}
and
\begin{equation}\label{bound:for:W} 
\max_{i=1,\ldots,r}\{\norm{ W_i }\} \simeq_{ C_2}\norm{Z}.
\end{equation}
The choice of the basis $\bar X_1,\dots,\bar X_{r}$ lets us identify $G/\brkt{G}{G}$ with $\R^r$ via the linear isomorphism $\phi:\R^r \to G/\brkt{G}{G}$ defined by
\[ \phi\left(\xi_{1},\dots,\xi_{r} \right):=\exp(\xi_{1}\bar X_1+\dots+\xi_{r}\bar X_r + \g^2), \]
where $\g^2=V_{2}\oplus\dots\oplus V_{s}$ and so $\exp( \g^2 )=\brkt{G}{G}$.
As a linear isomorphism, for some $C_3>1$, the map $\phi$ is a $C_3$-biLipschitz equivalence between $\R^r$ with the standard metric and $G/\brkt{G}{G}$ with the quotient metric. Consequently, we have
\begin{equation}\label{bilip:width}
\Width(a_1, \ldots, a_r)\simeq_{C_3} \Width(\phi(a_1), \ldots,\phi( a_r)) \quad \forall a_1, \ldots, a_r \in \R^r.
\end{equation}
We now show that the constant $K:=rC_1^2C_2 C_3$ satisfies the conclusion of the lemma. Take an arbitrary $Z \in V_{s}$ and write it as in \eqref{solution:W} for some $W_1,\dots,W_{r}\in V_{s-1}$ satisfying the bound \eqref{bound:for:W}.

Given points $x_0, \ldots, x_r\in G$ with $\Size(x_0,\dots,x_r)>0$, let $v_0, \ldots, v_r\in \R^r$ be such that $\phi( v_j) = \horproj(x_j)$ and write $v_j=\left(v_{j,1},\dots,v_{j,r} \right)$.
In other words, 
\[ x_j\in \exp\Big( \sum_{k=1}^{r} v_{j,k} \bar X_k + \g^2\Big). \]
Let $A$ be the $r\times r$ matrix whose $j$-th column is $A_j := v_{j}-v_{j-1}$,
so that
\begin{equation}\label{eq:A matrix defn}
	x_{j-1}^{-1} x_j \in \exp\Big( \sum_{k=1}^{r}(v_{j,k}-v_{j-1,k})\bar X_k + \g^2\Big)
	=\exp\Big( \sum_{k=1}^{r} \matelem{A}{k}{j} \bar X_k + \g^2\Big).
\end{equation}
The bound \eqref{bilip:width} combined with linearity of $\phi$ implies that the quantity ${\Width(A_1,\dots,A_r)}$ is comparable to $\Size(x_0, \ldots, x_r)$:
\begin{alignat*}{2}
\Width(A_1,\dots,A_r) &=& &\Width(v_{1}-v_{0},\dots,v_{r}-v_{r-1})\\
&\simeq_{C_3}& &\Width(\phi(v_{1}-v_{0}),\dots,\phi(v_{r}-v_{r-1}))\\
&=& &\Width(\phi(v_{1})-\phi(v_{0}),\dots,\phi(v_{r})-\phi(v_{r-1}))\\
&=& &\Width(\horproj (x_{1})-\horproj (x_{0}),\dots,\horproj (x_{r})-\horproj (x_{r-1}))\\
&=& &\Size(x_0, \ldots, x_r).
\end{alignat*}
In particular, $\Width(A_1,\dots,A_r)>0$ so we further deduce by Lemma~\ref{lemma:Cramer:width} that $A$ is invertible and its inverse $B$ satisfies
\begin{equation}\label{201801102} 
\abs{\matelem{B}{j}{l}}\leq \frac{1}{\Width(A_1,\dots,A_r)} 
\leq\frac{C_3}{\Size(x_0, \ldots, x_r)}.
\end{equation} 
Set $Y_j:=\sum_{l=1}^{r}\matelem{B}{j}{l}W_l$.
We shall verify that this choice of $Y_j$'s satisfies the conclusion of the lemma, i.e., the properties \eqref{lemma:correcting1} and \eqref{lemma:correcting2}.
	
The first property is deduced from bilinearity of the Lie bracket and the fact that $AB$ is the identity matrix. By \eqref{eq:A matrix defn}, we can write the vectors $X_j$ as sums
\[ X_j = \log( x_{j-1}^{-1} x_j) =\sum_{k=1}^{r} \matelem{A}{k}{j}\bar X_k + \mathfrak{g}^2.  \]
Since $\brkt{V_{s-1}}{\mathfrak{g}^2} = \brkt{V_{s-1}}{V_2\oplus\dots\oplus V_s} = 0$, it follows by bilinearity of the bracket that
\[ \sum_{j=1}^{r}\brkt{Y_j}{X_j} = \sum_{j=1}^{r}\brkt{\sum_{l=1}^{r}\matelem{B}{j}{l}W_l}{\sum_{k=1}^{r} \matelem{A}{k}{j}\bar X_k} = \sum_{k=1}^{r} \sum_{l=1}^{r} \sum_{j=1}^{r} \matelem{A}{k}{j} \matelem{B}{j}{l} \brkt{W_l}{\bar X_k}. \]
Using the fact that $AB$ is the identity matrix, we have $\sum_{j=1}^{r}\matelem{A}{k}{j} \matelem{B}{j}{l} = \delta_{kl}$, so the sum simplifies to
\[ \sum_{j=1}^{r}\brkt{Y_j}{X_j} = \sum_{k=1}^{r}\brkt{W_k}{\bar X_k} \overset{\eqref{solution:W}}{=} Z, \]
showing property~\eqref{lemma:correcting1}.

Regarding, property~\eqref{lemma:correcting2}, we first observe that estimating each $\norm{W_l}$ by \eqref{bound:for:W} and each $\abs{\matelem{B}{j}{l}}$ by \eqref{201801102}, we can bound $\norm{Y_j}$ by
\begin{align*}
\norm{Y_j} = \norm{\matelem{B}{j}{l}W_l} &\leq \sum_{l=1}^{r}\abs{\matelem{B}{j}{l}}\norm{W_l}
\\&\leq \sum_{l=1}^{r} \frac{C_2C_3}{\Size(x_0,\dots,x_r)}\norm{Z}
\\&=\frac{rC_2C_3}{\Size(x_0,\dots,x_r)}\norm{Z}.
\end{align*}
Then, using \eqref{normCuno} to give bounds for $\norm{Y_j}$ and $\norm{Z}$, we conclude that
\begin{align*}
C_1^{-1}d(\idelem,\exp(Y_j))^{s-1} \leq \norm{Y_j} &\leq \frac{rC_2C_3}{\Size(x_0,\dots,x_r)}\norm{Z} 
\\&\leq \frac{rC_1C_2C_3}{\Size(x_0,\dots,x_r)}d(\idelem,\exp(Z))^{s}.
\end{align*}
Hence the lemma holds with the proposed constant $K=rC_1^2C_2 C_3$.
\end{proof}

As mentioned before, the following lemma describes our error correction procedure. The strategy is the same as used before in \cite{Leonardi-Monti,Hakavuori_LeDonne_2016,Monti_Pigati_Vittone}. The geometric idea is that given a horizontal curve we perturb it adding an amount of length that depends on two factors: 
\begin{enumerate}[label=(\roman*)]
	\item the desired change ($k\in G$) in the endpoint of the curve, and
	\item the size of configuration of points ($x_0,\dots,x_r\in G$) that the curve passes through.
\end{enumerate}
However, instead of writing the argument using the language of curves, we write it as a form of a triangle inequality. The horizontal curve should be thought of as replaced by the points $x_0,\dots,x_r$ along the curve.
The benefits of this approach are twofold. First, we avoid having to worry about some technicalities, such as the parametrization of the curve or the concept of inserting one curve within another. Second, a triangle-inequality form is well suited to large-scale geometry, where the local behavior of horizontal curves is irrelevant. This allows us to immediately apply our argument in the asymptotic case not only to geodesics, but to rough geodesics as well.

\begin{lemma}\label{lemma:perturbation lemma}
	For every Carnot group $G$ of rank $r$ and step $s\geq 2$, there exists a constant $C>0$ with the following property:	
	
	Let $x_0, \ldots, x_r\in G$ and $k\in \exp(V_{s})$. If $\Size(x_0, \ldots, x_r)>0$, then
	\[ d(x_0, k x_r) \leq C\left(\frac{d(\idelem,k)^s}{\Size(x_0,\ldots,x_r)}\right)^{\frac{1}{s-1}}+ \sum_{j=1}^{r} d( x_{j-1}, x_j). \]
\end{lemma}
\begin{proof}
Let $K$ be the constant from Lemma~\ref{lemma:correcting system} for the group $G$. We claim that the constant $C:=2(r+1)K^{\frac{1}{s-1}}$ will satisfy the statement of the current lemma.
Given $x_0, \ldots, x_r\in G$ and $k\in \exp(V_{s})$, we apply Lemma~\ref{lemma:correcting system} with $Z:=\log(k)$ and $X_j:= \log( x_{j-1}^{-1} x_j)$, for $j=1,\ldots, r$.
We get the existence of $Y_1,\dots,Y_{r}\in V_{s-1}$ satisfying \eqref{lemma:correcting1} and the bound \eqref{lemma:correcting2}. 

Define the following points in $G$:
\[ \begin{array}{cclcc}
&&y_j:=\exp(Y_j), \quad \text{for }j=1,\ldots, r;&&\\
\alpha_0:=x_0, && \alpha_j := x_{j-1}^{-1} x_j , \quad \text{for }j=1,\ldots, r;&&\\
\beta_0:=y_1 ,&& \beta_j: = y_{j-1}^{-1} y_j , \quad \text{for }j=1,\ldots, r-1 ,&& \beta_r:= y_r^{-1}. \\
\end{array} \]
Since $Y_j\in V_{s-1}$, by the BCH formula we have 
\[ \mathrm{C}_{y_j}( \alpha_j ) = y_j \alpha_j y_{j }^{-1} = 
y_j \alpha_j y_{j }^{-1} \alpha_j^{-1} \alpha_j 
= \exp(\brkt{Y_j}{X_j}) \alpha_j, \]
where $\mathrm{C}_{y}$ denotes the conjugation by $y$.
Consequently, since $\exp(\brkt{Y_j}{X_j})\in \exp(V_s)$ commutes with everything, we have
\begin{align*}
\prod_{j=0}^r(\alpha_j\beta_j)
\nonumber&=\alpha_0 \beta_0 \alpha_1 \beta_1 \alpha_2 \beta_{2}\cdots \alpha_r \beta_{r}
\\\nonumber&=
\alpha_0 y_1 \alpha_1 y_{1}^{-1} y_2 \alpha_2 y_{2}^{-1}\cdots y_r \alpha_r y_{r}^{-1}
\\\nonumber&=
\alpha_0 \mathrm{C}_{y_1}( \alpha_1 ) \mathrm{C}_{ y_2}( \alpha_2)\cdots \mathrm{C}_{y_r}( \alpha_r)
\\&=
\alpha_0 \exp(\brkt{Y_1}{X_1}) \alpha_1 \exp(\brkt{Y_2}{X_2}) \alpha_2 \cdots \exp(\brkt{Y_r}{X_r}) \alpha_r
\\\nonumber&=
\exp(\brkt{Y_1}{X_1}) \exp(\brkt{Y_2}{X_2})\cdots \exp(\brkt{Y_r}{X_r}) \alpha_0 \alpha_1 \alpha_2 \cdots \alpha_r.
\end{align*}
Observe that a product of exponentials is the exponential of a sum for elements in $V_s$ and that the points $\alpha_j$ form the telescopic product $x_r = \alpha_0 \alpha_1 \alpha_2 \cdots \alpha_r$.
Thus the above identity simplifies to
\begin{equation}\label{eq:alphabeta product}
	\prod_{j=0}^r(\alpha_j\beta_j) 
	=\exp(\brkt{Y_1}{X_1}+\brkt{Y_2}{X_2}+\ldots + \brkt{Y_r}{X_r}) x_r
	\overset{\eqref{lemma:correcting1}}{=}\exp(Z)x_r=kx_r.
\end{equation}
By the definition of the points $\alpha_j$ for $j=1,\ldots,r$, we have
\begin{equation}\label{eq:dist alpha}
	d(\idelem,{\alpha_j}) = d( x_{j-1}, x_j),
\end{equation}
and for the points $\beta_j$ for $j=0,\ldots, r$, we have from \eqref{lemma:correcting2} the distance estimate
\begin{equation}\label{eq:dist beta}
	d(\idelem,{\beta_j}) \leq 2K^{\frac{1}{s-1}}\left(\frac{d(\idelem,k)^s}{\Size(x_0,\dots,x_r)}\right)^{\frac{1}{s-1}}.
\end{equation}
Combining \eqref{eq:alphabeta product}, \eqref{eq:dist alpha} and \eqref{eq:dist beta} we have that
\begin{eqnarray*}
d(x_0, k x_r) &\overset{\eqref{eq:alphabeta product}}{=}&
d(x_0, \Pi_{j=0}^r(\alpha_j\beta_j)) 
\\&=&
d(\idelem, \beta_0\Pi_{j=1}^r(\alpha_j\beta_j))
\\&\leq& d(\idelem,{\beta_0}) + \sum_{j=1}^r d(\idelem,{\beta_j} ) + \sum_{j=1}^r d(\idelem,{\alpha_j})
\\&\overset{{\eqref{eq:dist alpha} \& \eqref{eq:dist beta}}}{\leq}& 2(r+1)K^{\frac{1}{s-1}}\left(\frac{d(\idelem,k)^s}{\Size(x_0,\dots,x_r)}\right)^{\frac{1}{s-1}}+\sum_{j=1}^r d( x_{j-1}, x_j).
\end{eqnarray*}
Hence the lemma holds with the proposed constant $C=2(r+1)K^{\frac{1}{s-1}}$.
\end{proof}

The following proposition contains the particular form of triangle inequality that allows us to deduce our results for both tangents and asymptotic cones of geodesics. For any set of points $x_0,\dots,x_m\in G$ the standard triangle inequality states that
\[ d(x_0,x_m) \leq \sum_{k=1}^{m}d(x_{k-1},x_k). \]
The following proposition states that we can replace one of the terms of the sum with the distance $d(\stepproj(x_{\ell-1}), \stepproj(x_\ell))$ in the quotient group $G/\exp(V_s)$, if we pay a correction term coming from Lemma~\ref{lemma:perturbation lemma}.

Theorem~\ref{thm:tangent} for tangents will follow from the numerator of the correction term being related to the removed distance with a power $1+\epsilon$, which implies that in the tangential limit, the correction term is irrelevant. Theorem~\ref{thm:geodesic blowdowns} on the other hand will follow from the correction term being inversely related to the size of the configuration of the other points. This will allow us to apply Lemma~\ref{Euclidean:lemma} to constrain the behavior of geodesics on the large scale.

\begin{proposition}\label{prop:lower step triangle inequality}
	For every Carnot group $G$ of rank $r$ and step $s\geq 2$, there exists a constant $K>0$ such that for any $\configuration=(y_0, \ldots, y_{r+2})\in G^{r+3}$, $\ell\in\{1,\dots,r+2\}$ and $\configuration_\ell:=(y_0, \ldots,\hat{y}_{\ell-1},\hat{y}_{\ell},\ldots y_{r+2})\in G^{r+1}$
	the following modified triangle inequality holds:
	\[ d(y_0, y_{r+2}) \leq d(\stepproj(y_{\ell-1}), \stepproj(y_{\ell})) + K\left(\frac{d(y_{\ell-1}, y_{\ell})^s}{\Size(\configuration_\ell)}\right)^{\frac{1}{s-1}}+ \sum_{j\neq \ell} d( y_{j-1}, y_{j}). \]
\end{proposition}
\begin{proof}
	Since the claim of the proposition is degenerate when $\Size(\configuration_\ell)=0$, we can assume that $\Size(\configuration_\ell)>0$.
	Let $C$ be the constant of Lemma~\ref{lemma:perturbation lemma} for the group $G$. We claim that the constant $K:=2^{\frac{s}{s-1}}C$ will satisfy the statement of the proposition.
	
	By Proposition~\ref{prop:quotient} there exists $h\in\exp(V_{s})$ such that 
	\begin{equation}\label{h:prop}
	d(y_{\ell-1},hy_{\ell}) = d(\stepproj(y_{\ell-1}),\stepproj(y_{\ell})).
	\end{equation}
	We consider the points $x_j:=y_j$ for $j<\ell-1$ and $x_j:=hy_{j+2}$ for $j\geq\ell-1$. Since translation by $h$ does not change the horizontal projection, 
	\[ \Size(x_0,\dots,x_r) = \Size(\configuration_\ell)>0. \]
	Applying Lemma~\ref{lemma:perturbation lemma} with $k:=h^{-1}$ and 
	the points $x_0,\ldots, x_r$, we obtain the estimate
	\begin{equation}\label{eq:perturbation}
		d(x_0, h^{-1} x_r)\leq C\left(\frac{d(\idelem,h^{-1})^s}{\Size(x_0,\dots,x_r)}\right)^{\frac{1}{s-1}} + \sum_{j=1}^{r} d(x_{j-1}, x_{j}).
	\end{equation}
	By the definition of the points $x_j$, for $j\neq \ell-1$, we have
	\[ d(x_{j-1}, x_{j}) = \begin{cases}
	d(y_{j-1},y_j),&\text{if }j<\ell-1\\
	d(hy_{j+1},hy_{j+2}),&\text{if }j>\ell-1
	\end{cases} \]
	so
	\[ \sum_{j<\ell-1} d(x_{j-1}, x_{j}) = \sum_{j<\ell-1} d(y_{j-1}, y_{j})\]
	and
	\[ \sum_{j>\ell-1} d(x_{j-1}, x_{j})=\sum_{j>\ell+1} d(y_{j-1}, y_{j}). \]
	For $j=\ell-1$ on the other hand, applying the identity \eqref{h:prop} through a triangle inequality, we have
	\begin{align*}
	d(x_{\ell-2},x_{\ell-1}) &=
	d(y_{\ell-2},hy_{\ell+1}) 
	\\&\leq d(y_{\ell-2},y_{\ell-1}) + d(y_{\ell-1},hy_{\ell}) + d(hy_{\ell},hy_{\ell+1})
	\\&=d(y_{\ell-2},y_{\ell-1}) + d(\stepproj(y_{\ell-1}),\stepproj(y_{\ell})) + d(y_{\ell},y_{\ell+1}),
	\end{align*}
	filling in the missing terms $d(y_{j-1}, y_{j})$ for $j=\ell-1$ and $j=\ell+1$. Combining the cases, we get the estimate
	\begin{equation}\label{eq:x distance sum}
		\sum_{j=1}^{r} d(x_{j-1}, x_{j})
		\leq d(\stepproj(y_{\ell}),\stepproj(y_{\ell+1})) + \sum_{j\neq \ell}d(y_{j-1}, y_{j}).
	\end{equation}
	We combine the identity \eqref{h:prop}  with the fact that the projection $\stepproj$ is 1-Lipschitz, and we get  that $d(y_{\ell-1},hy_{\ell}) \leq d(y_{\ell-1},y_{\ell})$.
	Thus since $h$ is in the center of $G$, the distance $d(\idelem,h^{-1})$ can be estimated by
	\begin{equation}\label{eq:h estimate}
		d(\idelem,h^{-1}) 
		= d(hy_{\ell-1},y_{\ell-1})
		\leq d(hy_{\ell-1},hy_{\ell}) + d(hy_{\ell},y_{\ell-1})
		\leq 2d(y_{\ell-1},y_{\ell}).
	\end{equation}
	Combining \eqref{eq:perturbation} with \eqref{eq:x distance sum} and \eqref{eq:h estimate} results in the desired inequality
	\begin{align*}
	d(y_0,y_{r+2}) \leq 2^{\frac{s}{s-1}}C\left(\frac{d(y_{\ell-1},y_{\ell})^s}{\Size(x_0,\dots,x_r)}\right)^{\frac{1}{s-1}} &+ d(\stepproj(y_{\ell-1}),\stepproj(y_{\ell})) \\&+  \sum_{j\neq \ell}d(y_{j-1}, y_{j}).\qedhere
	\end{align*}
\end{proof}

\subsection{Geometric lemmas about minimal height and size}
None of the estimates of the rest of this section will be used for Theorem~\ref{thm:tangent}, so the reader interested in just the results about tangents can skip the following two lemmas.
For the proof of Theorem~\ref{thm:geodesic blowdowns} (and its generalization Theorem~\ref{thm:quasigeod}) we need to describe how the boundedness of the previously defined notions of $\Size$ and $\Width$ relate to uniform neighborhoods of hyperplanes in the abelianization $G/\brkt{G}{G}$. 
Lemma~\ref{Euclidean:lemma} describes how $\Width$ and hyperplane neighborhoods are related and Lemma~\ref{lemma:size lower bound} gives a lower bound for $\Size$ in terms of $\Width$ of a translation of the vertices.

We will only need the implications and estimates in one direction, however all of these lemmas can be generalized to include also the opposite inequalities (with possibly worse constants) and the reverse implications. 

\begin{lemma}\label{Euclidean:lemma}	
	Let $\Gamma$ be a subset of $\R^r$. If there exists $K>0$ such that $\Width(P)\leq K$ for all $P\in \Gamma^{m}$,
	then there exists an $(m-1)$-plane $W\subset \R^r$ such that $\Gamma\subset \bar B_{\R^r}(W,K)$. 
\end{lemma}
\begin{proof}
	Consider first the case when $\Gamma$ is a finite set. We take $P^*\in \Gamma^{m-1}$ so that the parallelotope $\PP(P^*)$ generated by $P^*$ maximizes the value $\vol_{m-1} \PP(P')$ among all $P' \in \Gamma^{m-1}$.
	We claim that $$\Gamma\subset \bar B_{\R^r}(\Span(P^*),K).$$ 
	Indeed, for every $a\in \Gamma$, since $P^*$ has maximal volume, we have by Remark~{\ref{rmk:width:volume and width}} that
	\[ d(a, \Span (P^*) ) 
	=\frac{\vol_{m} \PP( P^*,a ) }{ \vol_{m-1} \PP(P^*) } 
	\overset{{\ref{rmk:width:volume and width}}}{=}
	\Width(P^*,a)
	\leq K. \]
	Consider then the case of an infinite set $\Gamma$, and let $(p_n)_{n\in\N}$  be a countable dense set in $\Gamma$. Applying the lemma for the finite sets $\{p_1,\dots,p_n\}$, we have the existence of $(m-1)$-planes $W_n$ such that $\{p_1,\dots, p_n\}\subset \bar B_{\R^r}(W_n,K)$.
	By compactness there exist an $(m-1)$-plane $W$ and a diverging sequence $n_j$ such that $W_{n_j}\to W$, as $j\to\infty$.
	
	We want to prove that $\Gamma\subset \bar B_{\R^r}(W,K)$. 
	It is enough to show that $\{p_1,\dots, p_n\}\subset \bar B_{\R^r}(W,K+\eps)$, for all $n\in \N$ and $\eps>0$. 
	Fix such $n$ and $\eps$ and fix $R_n$ so that $\{p_1,\dots, p_n\}\subset \bar B_{\R^r}(0, R_n)$.
	Then we take $j$ large enough that $n_j>n$ and 
	\[ \bar B_{\R^r}(W_{n_j},K) \cap \bar B_{\R^r}(0, R_n) \subset \bar B_{\R^r}(W,K+\eps), \]
	which is possible since $W_{n_j} \to W$, and so $\bar B_{\R^r}(W_{n_j},K) \to \bar B_{\R^r}(W,K)$ on compact sets in the Hausdorff sense.
	Thus we conclude the proof of the claim:	
	\begin{align*}
	\{p_1,\dots, p_n\}
	\subset
	\{p_1,\dots, p_{n_j}\} \cap \bar B_{\R^r}(0, R_n)
	&\subset
	\bar B_{\R^r}(W_{n_j},K) \cap \bar B_{\R^r}(0, R_n)
	\\&\subset
	\bar B_{\R^r}(W,K+\eps). \qedhere
	\end{align*}
\end{proof}
For convenience of applying Lemma~\ref{Euclidean:lemma} within the proof of Theorem~\ref{thm:geodesic blowdowns}, we give a lower bound for $\Size$ in terms of $\Width$. We will not need this bound for the proof of Theorem~\ref{thm:tangent}.
\begin{lemma}\label{lemma:size lower bound}
	In any Carnot group $G$, there exists a constant $c>0$ such that the following holds:
	
	For any $\configuration=(g_0,\dots,g_r)\in G^{r+1}$ and $\ell\in\{0,\dots,r\}$, let $\Gamma_\ell\in (G/\brkt{G}{G})^{r}$ be the tuple of the points $\horproj(g_j)-\horproj(g_\ell)$, $j\neq \ell$. Then
	\[ \Size(\configuration) \geq c\cdot\Width(\Gamma_\ell). \]
\end{lemma}
\begin{proof}
	In $\R^n$, consider for each $\ell\in\{0,\dots,r\}$ the map $A^\ell:(\R^n)^r\to (\R^n)^r$, whose component functions $A^\ell_k:(\R^n)^r\to \R^n$ are defined by
	\[ A^\ell_k(x_1,\dots,x_r) = \sum_{j=k}^{\ell}x_j\quad\text{for }k=1,\dots,\ell \]
	and by
	\[ A^\ell_k(x_1,\dots,x_r) = \sum_{j=\ell+1}^{k}x_j\quad\text{for }k=\ell+1,\dots,r. \]
	In block-matrix form, the linear map $A^\ell$ has the form $A^\ell = \begin{bmatrix}U&0\\0&L\end{bmatrix}$, where  
	\[ U = \begin{bmatrix}
	I&I&\dots&I\\
	0&I&\dots&I\\
	\vdots&\vdots&\ddots&\vdots\\
	0&0&\dots&I
	\end{bmatrix}
	\quad\text{and}\quad L = \begin{bmatrix}
	I&0&\dots&0\\
	I&I&\dots&0\\
	\vdots&\vdots&\ddots&\vdots\\
	I&I&\dots&I
	\end{bmatrix} \]
	are themselves $\ell\times\ell$ and $(r-\ell)\times(r-\ell)$ upper and lower triangular block-matrices, whose $n\times n$-blocks are all either the $n\times n$ identity matrix $I$ or zero.
	
	From the above description, it is clear that $A^\ell$ is a linear bijection, so there exists a constant $C_\ell>0$ such that $A^\ell$ is a $C_\ell$-biLipschitz map. Thus for any set $\PP\subset \R^r$, we have 
	\[ C_\ell^{-m}\vol_m(\PP)\leq \vol_m(A^\ell(\PP)) \leq C_\ell^m\vol_m(\PP). \]
	By the characterization of $\Width$ as volume quotients in Remark~\ref{rmk:width:volume and width}, it follows that
	\begin{equation}\label{eq:bilip width distortion}
	\Width(A^\ell(x_1,\dots,x_r))
	\leq C_\ell^{2r-1}\cdot\Width(x_1,\dots,x_r)
	\end{equation}
	The abelianization $G/\brkt{G}{G}$ is a normed space, so there exists for some $C>0$ and $n\in\N$ a $C$-biLipschitz isomorphism $\phi:G/\brkt{G}{G}\to\R^n$.
	We claim that the constant 
	\begin{equation}\label{const:c}
	c := \min_{\ell\in\{0,\dots,r\}} C^{-2}C_\ell^{1-2r}
	\end{equation}
	satisfies the claim of the lemma.
	
	Let $y_j:=\horproj(g_{j})-\horproj(g_{j-1})$, $j=1,\dots,r$ so that the definition \eqref{eqdef:size} of $\Size$ is written as
	\begin{equation}\label{eq:size of config}
		\Size(\configuration) = \Size(g_0,\dots,g_r) = \Width(y_1,\dots,y_r).
	\end{equation}
	Apply the map $(\phi^{-1})^r\circ A^\ell\circ (\phi)^r:(G/\brkt{G}{G})^r\to (G/\brkt{G}{G})^r$ to the tuple $(y_1,\dots,y_r)\in (G/\brkt{G}{G})^r$. For $k\leq \ell$, we have
	\begin{align*}
	(\phi^{-1})^r\circ A^\ell_k(\phi(y_1),\dots,\phi(y_r)) 
	&= (\phi^{-1})^r\left(\sum_{j=k}^{\ell}\left(\phi\circ\horproj(g_{j})-\phi\circ\horproj(g_{j-1})\right) \right)
	\\&= (\phi^{-1})^r\left(\phi\circ\horproj(g_\ell)-\phi\circ\horproj(g_{k-1}) \right)
	\\&= \horproj(g_\ell) - \horproj(g_{k-1}).
	\end{align*}
	Similarly for $k\geq \ell+1$, we have
	\[ (\phi^{-1})^r\circ A^\ell_k(\phi(y_1),\dots,\phi(y_r)) = \horproj(g_k)-\horproj(g_{\ell}). \]
	That is, up to the sign of the elements $k\leq \ell$ components, the components of $(\phi^{-1})^r\circ A^\ell(\phi(y_1),\dots,\phi(y_r))$ form exactly the tuple $\Gamma_\ell$.
	
	For any $C$-Lipschitz map $f$, we have
	\[ \Width(f(y_1),\dots,f(y_r)) \leq C\cdot\Width(y_1,\dots,y_r). \]
	Since both $\phi$ and $\phi^{-1}$ are $C$-Lipschitz, by \eqref{eq:bilip width distortion} we get
	\begin{eqnarray*}
	\Width(\Gamma_\ell) 
	&=& \Width((\phi^{-1})^r\circ A^\ell(\phi(y_1),\dots,\phi(y_r)))
	\\&\leq &C\cdot\Width(A^\ell(\phi(y_1),\dots,\phi(y_r)))
	\\&\overset{\eqref{eq:bilip width distortion}}{\leq}& CC^{2r-1}_{\ell}\cdot\Width(\phi(y_1),\dots,\phi(y_r))
	\\&\leq& C^2C^{2r-1}_{\ell}\cdot\Width(y_1,\dots,y_r).
	\end{eqnarray*}
	By \eqref{eq:size of config} and \eqref{const:c} we end up with the desired estimate
	\begin{align*}
	\Size(\configuration)\overset{\eqref{eq:size of config}}{=}\Width(y_1,\dots,y_r) &\geq \frac{1}{C^2C_\ell^{2r-1}}\Width(\Gamma_\ell)
	\\&\overset{\eqref{const:c}}{\geq} c\cdot\Width(\Gamma_\ell).\qedhere
	\end{align*}
\end{proof}

\section{Blowups of geodesics}\label{sec:UP}
We next prove the results on blowups (i.e.\ tangents) of geodesics (Theorems \ref{thm:sr tangent iteration} and \ref{thm:tangent}).
In fact, instead of the qualitative claim of Theorem~\ref{thm:tangent}, we will prove a slightly stronger quantified statement. We show that $\stepproj\circ\gamma$ satisfies a sublinear distance estimate on some small enough interval, implying that any tangent of $\stepproj\circ\gamma$ is a geodesic. The estimate shall follow by applying the triangle inequality of Proposition~\ref{prop:lower step triangle inequality} with tuples $E=(y_0,\dots,y_{r+2})$ where only two of the points $y_{\ell-1}$ and $y_\ell$ will vary.

The quantified claim of Theorem~\ref{thm:quantified tangent} allows us to prove Theorem~\ref{thm:tangent} and Corollary~\ref{cor:tangent iteration} also for the larger family of \emph{weak tangents} (see Theorem~\ref{thm:weak tangent} and Corollary~\ref{cor:weak tangent iteration}), where the fixed basepoint $\bar{t}$ of the blowup is replaced with a converging sequence $t_j\to\bar{t}$.
We denote by $\gamma_{h,\bar{t}}: I_{h,t} \to G$ the curve defined on $I_{h,\bar{t}} := \tfrac{1}{h} (I - \bar{t})$ by
\begin{equation}\label{eq:weak tangent}
\gamma_{h,\bar{t}}(t) := \delta_{\frac{1}{h}} \left( \gamma(\bar t)^{-1} \gamma( \bar t + ht )\right).
\end{equation}
The collection of all weak tangents will be denoted by
\[ \WeakTang(\gamma,\bar t) := \left\{ \sigma\; \vert\; \exists h_j\to 0, t_j\to\bar{t} :\gamma_{h_{j},t_j}\to \sigma \right\} \]
and the iterated weak tangent cones are denoted by
\[ \WeakTang^{k+1}(\gamma,\bar t) := \bigcup_{\sigma\in\Tang^k(\gamma,\bar t)}\bigcup_{t\in\R}\WeakTang(\sigma,t). \]

\begin{theorem}\label{thm:quantified tangent}
	Let $G$ be a Carnot group of step $s$ and let $\gamma:I\to G$ be a geodesic. Then for any $\bar{t}\in I$, there exist constants $C>0$ and $\delta>0$ such that for all $a,b\in (\bar{t}-\delta,\bar{t}+\delta)$,
	\[ \abs{a-b}-C\abs{a-b}^{\frac{s}{s-1}}\leq d(\stepproj\circ\gamma(a),\stepproj\circ\gamma(b))\leq \abs{a-b}. \]
\end{theorem}
\begin{proof}
	The upper bound follows directly from the projection $\stepproj:G\to G/\exp(V_s)$ being 1-Lipschitz. The non-trivial statement is the lower bound, which will follow from Proposition~\ref{prop:lower step triangle inequality}.
	
	Translating the parametrization if necessary, we may assume that $\bar{t}=0$. Since any geodesic is still a geodesic within every Carnot subgroup containing it, we may also assume that $G$ is the smallest Carnot subgroup containing $\gamma(I)$. Hence, if $r$ is the rank of $G$, there exist $t_0,\dots,t_r\neq 0$ such that the points $\horproj\circ\gamma(t_0),\dots,\horproj\circ\gamma(t_r)$ are in general position. 
	By Remark~{\ref{rmk:width:general position and size}}, we have that
	\begin{equation}\label{cond:delta}
	\Delta:=\Size( \gamma(t_0),\dots, \gamma(t_r))  > 0.
	\end{equation}
	Let $K$ be the constant given by Proposition~\ref{prop:lower step triangle inequality} for the Carnot group $G$. We claim that the constants $C:=K\Delta^{-\frac{1}{s-1}}$ and $\delta:=\min(\abs{t_0},\dots,\abs{t_r})$ will satisfy the claim of the theorem.
	
	Fix $a,b\in (-\delta,\delta)$. 
	Consider the set of points 
	\[ \configuration:=\{y_0,\dots,y_{r+2}\} = \{\gamma(t_j):\; j=0,\dots,r \}\cup \{\gamma(a),\gamma(b)\}, \]
	where the points $y_j$ are indexed by the order in which they appear along $\gamma$. By the choice of $\delta$, the points $\gamma(a)$ and $\gamma(b)$ are consecutive in this ordering, so there is some $\ell\in\{1,\dots,r+2\}$ such that $y_{\ell-1} = \gamma(a)$ and $y_\ell = \gamma(b)$.
	
	We apply Proposition~\ref{prop:lower step triangle inequality} with the above $\configuration$ and $\ell$. By \eqref{cond:delta}, we get the estimate
	\begin{equation}\label{eq:tangent cut and correct}
	d(y_0, y_{r+2}) \leq d(\stepproj\circ\gamma(a), \stepproj\circ \gamma(b))
	+ K\left(\frac{d(\gamma(a), \gamma(b))^s}{\Delta}\right)^{\frac{1}{s-1}}
	+ \sum_{j\neq \ell} d( y_{j-1}, y_{j}).
	\end{equation}
	By the choice of the points $y_j$ as sequential points along the geodesic $\gamma$, we have
	\begin{equation}\label{eq:sequential points distance}
		\sum_{j\neq \ell} d( y_{j-1}, y_{j}) = d(y_0,y_{r+2}) - d( y_{\ell-1}, y_{\ell}) = d(y_0,y_{r+2}) - d(\gamma(a), \gamma(b)).
	\end{equation}
	We then apply the identity \eqref{eq:sequential points distance} to \eqref{eq:tangent cut and correct}, we use the fact that $\gamma|_{[a,b]}$ is a geodesic, and we reorganize the terms. This gives the lower bound
	\[ d(\stepproj\circ\gamma(a), \stepproj\circ \gamma(b)) 
	\geq \abs{a-b}
	-K\Delta^{-\frac{1}{s-1}}\abs{a-b}^{\frac{s}{s-1}}, \]
	proving the claim of the theorem.
\end{proof}
 
Theorem~\ref{thm:tangent} shall follow immediately from Theorem~\ref{thm:quantified tangent} by taking any limit of dilations $h_k\to 0$. We prove the result for the larger family of all weak tangents.

\begin{theorem}\label{thm:weak tangent}
	If $\gamma:I\to G$ is a geodesic and $t\in I$,
	then for every $\sigma\in\WeakTang(\gamma,t)$, the curve $\stepproj\circ\sigma:\R\to G/\exp(V_s)$ is a geodesic.
\end{theorem}
\begin{proof}
	Reparametrizing and left-translating if necessary, we may assume that $t=0$ and $\gamma(0)=\idelem$. 
	Then $\sigma\in\WeakTang(\gamma,0)$ is given by some sequences $h_k\to 0$ and $t_k\to 0$ as $\sigma=\lim\limits_{k\to\infty}\gamma_{h_k,t_k}$.
	
	For any $h>0$, $t\in I$, and $a,b\in I_{h,t}$, expanding the definition \eqref{eq:weak tangent} of the dilated curve $\gamma_{h,t}$, we obtain the identity
	\begin{equation}\label{eq:dilated curve distance}
		d(\gamma_{h,t}(a),\gamma_{h,t}(b)) = \frac{1}{h}d(\gamma(ha),\gamma(hb)).
	\end{equation}
	Let $C>0$ and $\delta>0$ be the constants of Theorem~\ref{thm:quantified tangent} for the geodesic $\gamma$. Rephrasing the statement of Theorem~\ref{thm:quantified tangent} for $\gamma_{h,t}$ using \eqref{eq:dilated curve distance}, we get for all $a,b\in ((-\delta-t)/h,(\delta-t)/h)$ that
	\begin{equation}\label{eq:dilated sublinear distance estimate}
		\abs{a-b} - Ch^{\frac{1}{s-1}}\abs{a-b}^{\frac{s}{s-1}} \leq d(\stepproj\circ\gamma_{h,t}(a),\stepproj\circ\gamma_{h,t}(b))\leq\abs{a-b}.
	\end{equation}
	For any $a,b\in\R$, the condition $a,b\in ((-\delta-t_k)/h_k,(\delta-t_k)/h_k)$ is satisfied for any large enough indices $k\in\N$. Thus taking the limit of \eqref{eq:dilated sublinear distance estimate} as $h=h_k\to 0$, we get for the limit curve $\stepproj\circ\sigma=\lim\limits_{k\to\infty}\stepproj\circ\gamma_{h_k,t_k}$ the estimate
	\[ \abs{a-b}\leq d(\stepproj\circ\sigma(a),\stepproj\circ\sigma(b)) \leq \abs{a-b}, \]
	showing that $\stepproj\circ\sigma$ is a geodesic.
\end{proof}

By induction from Theorem~\ref{thm:weak tangent} we obtain the following corollary, which is the stronger form of Corollary~\ref{cor:tangent iteration}.
\begin{corollary}\label{cor:weak tangent iteration}
	If $\gamma:I\to G$ is a geodesic and $t\in I$,
	then for every $\sigma\in\WeakTang^{s-1}(\gamma,t)$, the horizontal projection $\horproj\circ\sigma$ is a geodesic.
	In particular, if $G$ is \subriemannian then every $\sigma\in\WeakTang^{s-1}(\gamma,t)$ is a line.
\end{corollary}

We prove next that Theorem~\ref{thm:sr tangent iteration} follows from the fact that the metric tangent of a \subriemannian manifold is a quotient of a \subriemannian Carnot group, which is a well known theorem attributed to Bella\"iche \cite{bellaiche}.

\begin{proof}[Proof of Theorem~\ref{thm:sr tangent iteration}]
	Let $M$ be a \subriemannian manifold. Let $s$ be the step of the \subriemannian manifold, i.e., Lie brackets of length $s$ of the horizontal vector fields in $M$ span the tangent spaces $T_pM$ at each point $p\in M$. Let $\gamma:I\to M$ be a geodesic.
	
	Following \cite[Theorem 3.6]{Monti_Pigati_Vittone:tangent_cones}, we see that any metric tangent $\sigma$ of $\gamma$ is a geodesic in the nilpotent approximation $\tilde{M}$ of $M$. By \cite[Theorem 2.7]{jeancontrol}, the nilpotent approximation of a \subriemannian manifold is a homogeneous space $G/H$, where $G$ is a Carnot group of step $s$ and $H<G$ is a closed dilation invariant Lie subgroup. In particular, any iterated tangent of $\sigma$ gives another geodesic in $\tilde{M}=G/H$.
	
	On the other hand, since the projection $\pi:G\to G/H$ is a submetry, the geodesic $\sigma$ can be lifted to a geodesic $\tilde{\sigma}$ in $G$.
	Applying Corollary~\ref{cor:tangent iteration} we see that any $s-1$ times iterated tangent of $\tilde{\sigma}$ is a line. Projecting back to $G/H$, we see that also necessarily any $s-1$ times iterated of $\sigma$ must be a line. Since $\sigma$ was an arbitrary tangent of $\gamma$, it follows that any $s$ times iterated tangent of $\gamma$ is a line.
\end{proof}

\section{Blowdowns of rough geodesics}\label{sec:DOWN}
In this section we prove Theorem~\ref{thm:geodesic blowdowns} and Corollary~\ref{cor:subriemannian blowdowns}. Due to our formulation of the core of the argument (Proposition~\ref{prop:lower step triangle inequality}) as a triangle inequality, we are able to prove the stronger claims of Theorem~\ref{thm:quasigeod} and Corollary~\ref{cor:subriemannian quasiblowdowns} for rough geodesics.

To make the terminology precise, by \emph{rough geodesic}, we mean a not-necessarily-continuous curve that is a $(1,C)$-quasi-geodesic for some $C\geq 0$.
By a \emph{$(1,C)$-quasi-geodesic} we mean a \emph{$(1,C)$-quasi-isometric embedding}, i.e., some $\gamma: I\to G$ such that 
\begin{equation}\label{eq:euclidean quasigeod:1,C}
	\abs{t_{1}-t_2} - C 
	\leq d(\gamma(t_1), \gamma(t_2))
	\leq \abs{t_{1}-t_2}+C, \qquad \forall t_1, t_2\in I.
\end{equation}
Thus a $(1,0)$-quasi-geodesic is exactly a geodesic.

\begin{theorem}\label{thm:quasigeod} 
If $\gamma:\mathbb{R}\to G$ is a $(1,C)$-quasi-geodesic, then one of the following holds:
\begin{enumerate}[label=(\arabic{section}.\arabic{equation}.\roman{*})]
	\item\label{enum:quasigeod:hyperplane} There exist a hyperplane $W\subset V_1$ and some $R>0$ such that $\image(\horproj\circ\gamma)\subset B_{V_1}(W,R)$.
	\item\label{enum:quasigeod:horizontal quasigeod} There exists $C'\geq 0$ such that $\horproj\circ\gamma:\mathbb{R}\to G/\brkt{G}{G}$ is a $(1,C')$-quasi-geodesic.
\end{enumerate}
Moreover, one can take $C' = (r+2)^{s-1}C$, where $r$ is the rank and $s$ is the step of the Carnot group $G$.
\end{theorem}
\begin{proof}
Assume \ref{enum:quasigeod:hyperplane} does not hold.
We claim that it is enough to show that $\stepproj\circ\gamma $ is a $(1, C_1)$-quasi-geodesic with $C_1:=(r+2)C$. Indeed, then we can iterate: the curve $\stepproj\circ\gamma$ has the same projection as $\gamma$ on $G/\brkt{G}{G}$. Thus, \ref{enum:quasigeod:hyperplane} does not hold for $\stepproj\circ\gamma$ either, and we have that $\multistepproj{s-2}\circ\stepproj\circ\gamma$ is a $(1,C_2)$-quasi-geodesic in $G/\exp(V_{s-1}\oplus V_s)$ with $C_2=(r+2)C_1=(r+2)^2C$. We repeat until after $(s-1)$ steps we get that $\horproj\circ\gamma=\multistepproj{1}\circ\cdots \circ\stepproj\circ\gamma$ is a $(1,(r+2)^{s-1}C)$-quasigeodesic.

As with Theorem~\ref{thm:quantified tangent}, the upper bound follows immediately from the projection $\stepproj$ being $1$-Lipschitz. Thus it is enough to show the lower bound $\abs{b-a}-C_1\leq d(\stepproj\circ\gamma(a),\stepproj\circ\gamma(b))$, for all $a,b\in \R$.

Set $\Gamma:=\gamma({\R\setminus[a,b]})$ and fix an arbitrary basepoint $\bar{t}\in \R\setminus[a,b]$. Since \ref{enum:quasigeod:hyperplane} does not hold for $\gamma$, the same is true for any translation of $\gamma$. Therefore we can assume without loss of generality that $\gamma(\bar{t})=\idelem$.

Fix an arbitrary $\epsilon>0$. Let $K>0$ be the constant of Proposition~\ref{prop:lower step triangle inequality} and let $c>0$ be the constant of Lemma~\ref{lemma:size lower bound}. Since $\gamma([a,b])$ is a bounded set, the failure of \ref{enum:quasigeod:hyperplane} for $\gamma$ implies that $\Gamma$ is also not contained in any neighborhood of any hyperplane. Since $G/\brkt{G}{G}$ and $\R^r$ are biLipschitz equivalent, Lemma~\ref{Euclidean:lemma} implies that $\Width(\horproj(P))$ is not bounded as $P$ varies in $\Gamma^r$. In particular, we may fix some $P\in\Gamma^r$ such that
\begin{equation}\label{eq:P:large width}
	\Width(\horproj(P)) > \frac{K^{s-1}d(\gamma(a),\gamma(b))^s}{c\epsilon^{s-1}}.
\end{equation}
Consider the tuple $\configuration:=(\gamma(t_0),\dots,\gamma(t_{r+2}))$, where
\[ \{\gamma(t_0),\dots,\gamma(t_{r+2})\} = P\cup \{\gamma(\bar{t}),\gamma(a),\gamma(b)\}, \]
with the times $t_j$ ordered so that $t_0< \dots< t_{r+2}$. 

By the definition of $\Gamma$ and $\bar{t}$, the points $\gamma(a)$ and $\gamma(b)$ are necessarily consecutive in this ordering, so there is some $\ell\in\{1,\dots,r+2\}$ such that $t_{\ell-1}=a$ and $t_\ell = b$. Denote by $\configuration_P\in \Gamma^{r+1}$ the tuple $E$ without $\gamma(a)$ and $\gamma(b)$, i.e.,
\[ \configuration_P := (\gamma(t_0),\dots,\gamma(t_{\ell-2}),\gamma(t_{\ell+1}),\dots,\gamma(t_{r+2})). \]
Applying Proposition~\ref{prop:lower step triangle inequality} with the above $\configuration$ and $\ell$, we get the bound
\begin{align}\label{eq:asymptotic cut and correct}
d(\gamma(t_0), \gamma(t_{r+2})) 
\leq d(\stepproj\circ\gamma(a), \stepproj\circ \gamma(b))+ \sum_{j\neq \ell} d(\gamma(t_{j-1}),\gamma(t_j))&
\\\nonumber+ K\left(\frac{d(\gamma(a), \gamma(b))^s}{\Size(\configuration_P)}\right)^{\frac{1}{s-1}}.&
\end{align}
Estimating the distances along $\gamma$ by \eqref{eq:euclidean quasigeod:1,C} gives
\[ \sum_{j\neq \ell} d(y_{j-1}, y_{j}) 
\leq \sum_{j\neq \ell}\abs{t_{j-1}-t_j} + (r+1)C = \abs{t_0-t_{r+2}}-\abs{a-b}+(r+1)C \]
and
\[ d(\gamma(t_0), \gamma(t_{r+2})) \geq \abs{t_0-t_{r+2}} - C. \]
Applying the above distance estimates to \eqref{eq:asymptotic cut and correct} and reorganizing terms, we get the lower bound
\begin{equation}\label{eq:asymptotic lower bound}
	d(\stepproj\circ\gamma(a), \stepproj\circ \gamma(b)) \geq \abs{a-b} - (r+2)C - K\left(\frac{d(\gamma(a), \gamma(b))^s}{\Size(\configuration_P)}\right)^{\frac{1}{s-1}},
\end{equation}
which is exactly the desired lower bound except for the final term.

However, since $\gamma(\bar{t})=\idelem$, applying Lemma~\ref{lemma:size lower bound} with $\ell$ such that $t_\ell=\bar{t}$ gives
\begin{equation}\label{eq:combined size lower bound}
	\Size(\configuration_P) \geq c\cdot\Width(\horproj(P)).
\end{equation}
Bounding $\Size(\configuration_P)$ by \eqref{eq:combined size lower bound} and $\Width(\horproj(P))$ by \eqref{eq:P:large width}, the lower bound \eqref{eq:asymptotic lower bound} is simplified to
\[ d(\stepproj\circ\gamma(a), \stepproj\circ \gamma(b)) \geq \abs{a-b} - (r+2)C - \epsilon. \]
Since $\epsilon>0$ was arbitrary, we have the desired quasi-geodesic lower bound.
\end{proof}

The second possible conclusion \ref{enum:quasigeod:horizontal quasigeod} in Theorem~\ref{thm:quasigeod} is that $\horproj\circ\gamma$ is a quasi-geodesic in the normed space $G/\brkt{G}{G}$. We next show that in the case of an inner product space, quasi-geodesics are well behaved on the large scale. Namely, every rough geodesic in $\R^n$ has a unique asymptotic cone and this asymptotic cone is a line. The result can also be deduced from \cite[Theorem~2.5]{Rassias-Xiang-2002-stability-of-approximate-isometries}. Note that not all $(1,C)$-quasi-geodesics are a finite distance from a line, since for example graphs of $1/2$-Hölder functions $\R\to\R$ are $(1,C)$-quasi-geodesics in $\R^2$.

\begin{lemma}\label{Euclid:rough:geodesic}
Every $(1,C)$-quasi-geodesic $\gamma:\R\to \R^n$ in a Euclidean space $\R^n$ has a unique blowdown and the blowdown is a line.
\end{lemma}
\begin{proof}
Translating and reparametrizing if necessary, we may assume that $\gamma(0)=0$. 
We denote by $\angle(t,s)$ the angle formed by $\gamma(t)$ and $\gamma(s)$ at 0. Its magnitude is given by the standard inner product on $\R^n$ via
\begin{equation}\label{eq:euclidean quasigeod:angle}
\cos\angle(t,s) = \frac{\gamma(t)\cdot\gamma(s)}{\abs{\gamma(t)}\abs{\gamma(s)}}.
\end{equation}
We first show that as $t,s\to\infty$, the angle vanishes, i.e., $1-\cos\angle(t,s)\to 0$. By symmetry we can assume that $t\geq s\geq 0$. 

In an inner product space we have for all $x,y $ the identity
\[ 2\abs{x}\abs{y}-2x\cdot y = \abs{x-y}^2- (\abs{x}-\abs{y})^2. \]
Combining \eqref{eq:euclidean quasigeod:angle} and the above identity for $x=\gamma(t)$ and $y=\gamma(s)$, we get
\begin{align}\label{eq:angle estimate}
1-\cos\angle(t,s) &= \frac{2\abs{\gamma(t)}\abs{\gamma(s)}-2\gamma(t)\cdot\gamma(s)}{2\abs{\gamma(t)}\abs{\gamma(s)}}
\\\nonumber&=\frac{\abs{\gamma(t)-\gamma(s)}^2-(\abs{\gamma(t)}-\abs{\gamma(s)})^2}{2\abs{\gamma(t)}\abs{\gamma(s)}}.
\end{align}
The quasi-geodesic bound \eqref{eq:euclidean quasigeod:1,C} and the assumption $t\geq s\geq 0$ imply that
\begin{align*}
\abs{\gamma(t)-\gamma(s)}^2-(\abs{\gamma(t)}-\abs{\gamma(s)})^2
&\leq (t-s+C)^2-(t-s-2C)^2
\\&=6C(t-s)-3C^2
\leq 6Ct.
\end{align*}
Moreover, the bound \eqref{eq:euclidean quasigeod:1,C} implies also that when $t,s\geq 2C$ we have
\begin{align*}
\abs{\gamma(t)}\abs{\gamma(s)} 
&\geq (t-C)(s-C)
\geq \frac{1}{4}ts.
\end{align*}
Estimating \eqref{eq:angle estimate} using the above two inequalities, we get for all $t,s\geq 2C$ the upper bound
\[ 1-\cos\angle(t,s)\leq \frac{6t}{\frac{1}{4}ts}=\frac{24}{s} \]
and hence $\angle(t,s)\to 0$ as $t\geq s\to\infty$. Repeating a similar argument for $t\leq s\leq 0$, we see also that $\angle(t,s)\to 0$ as $t\leq s\to-\infty$.

From this estimate of angles we conclude that the limit directions $v_+=\lim\limits_{t\to\infty}\gamma(t)/\abs{\gamma(t)}$ and $v_-=\lim\limits_{t\to-\infty}\gamma(t)/\abs{\gamma(t)}$ always exist. We claim that this implies that the asymptotic cone $\lim\limits_{h\to\infty}\gamma_h$ exists without taking any subsequences, thus proving uniqueness.

First, observe that the existence of the limit direction and $\gamma$ being a quasi-geodesic implies that also $\lim\limits_{t\to\infty}\gamma(t)/t=v_+$. Indeed, for any $t>C$, by \eqref{eq:euclidean quasigeod:1,C}, we have
\[ \abs{\frac{\gamma(t)}{t}-\frac{\gamma(t)}{\abs{\gamma(t)}}} 
=\frac{\abs{\gamma(t)}\abs{\abs{\gamma(t)}-t}}{t\abs{\gamma(t)}}
\leq \frac{(t+C)C}{t(t-C)}\to 0 \]
as $t\to\infty$. This implies that $\lim\limits_{h\to\infty}\gamma_h(1)=v_+$. For arbitrary $t>0$,
\[ \lim\limits_{h\to\infty}\gamma_h(t) = \lim\limits_{h\to\infty}\frac{\gamma(ht)}{h} = t\lim\limits_{h\to\infty}\frac{\gamma(ht)}{ht}=tv_+. \]
Similarly $\lim\limits_{h\to\infty}\gamma_h(t)=-tv_-$ for all $t<0$, proving existence and uniqueness of the blowdown.

To see that the unique limit is a line, i.e., that $v_-=-v_+$, it suffices to observe that any blowdown of a $(1,C)$-quasi-geodesic in $\R^n$ is a geodesic in $\R^n$, and geodesics in $\R^n$ are lines.
\end{proof}

Combining Theorem~\ref{thm:quasigeod} with Lemma~\ref{Euclid:rough:geodesic} allows us to conclude the lower rank subgroup containment for blowdowns of rough geodesics in \subriemannian Carnot groups:
\begin{corollary}\label{cor:subriemannian quasiblowdowns}
If $\gamma: \R\to G$ is a $(1,C)$-quasi-geodesic in a \subriemannian Carnot group $G\neq \R$, then there exists a proper Carnot subgroup $H<G$
containing every element of $\Asymp(\gamma)$.
\end{corollary}
\begin{proof}
	Consider the two cases of Theorem~\ref{thm:quasigeod}.
	In the first case \ref{enum:quasigeod:hyperplane}, the horizontal projection is in a finite neighborhood of a hyperplane, $\image(\horproj\circ\gamma)\subset B_{V_1}(W,R)$. Thus any blowdown $\sigma\in\Asymp(\gamma)$ has its horizontal projection completely contained in $W$. Since $\sigma(0)=\idelem$, it follows that $\sigma$ is contained in the Carnot subgroup $H$ generated by $W$. The rank of $H$ is by construction the dimension of $W$, which is smaller than the rank of $G$.
	
	In the second case \ref{enum:quasigeod:horizontal quasigeod}, the horizontal projection $\pi\circ\gamma$ is a $(1,C')$-quasi-geodesic. Thus by Lemma~\ref{Euclid:rough:geodesic} it has a unique blowdown $\sigma$, which is a line. But then $H:=\sigma(\R)$ is itself a one-parameter subgroup containing all blowdowns, proving the claim.
\end{proof}

\section{Infinite geodesics}\label{sec:infinite geodesics}
\subsection{Abnormality of blowdowns of geodesics}
Let $G$ be a \subriemannian Carnot group, so that on the first layer $V_1$ of the Lie algebra $\g$ we have an inner product $\innerproduct{\cdot}{\cdot}$.
Every geodesic $\gamma:I\to G$ on a finite interval $I\subset\R$ is then a solution of the Pontryagin maximum principle. In \subriemannian Carnot groups the principle takes the form
\begin{equation}\label{PMP:Lie}\tag{PMP}
\lambda\left( \int_I \Ad_{\gamma (t)} v(t)\, dt\right)=\xi\innerproduct{u_\gamma}{v} \quad \forall v\in L^2(I;V_1),
\end{equation}
for some $\lambda \in \g^*$ and $\xi\in \mathbb R$ such that $(\lambda,\xi)\neq (0,0)$, see \cite{LMPOV} for the calculation of the differential of the endpoint map. Here, $u_\gamma\in L^2(I;V_1)$ denotes the control of $\gamma$. 

A curve is abnormal exactly when it satisfies \ref{PMP:Lie} with $\xi=0$ for some $\lambda\in\g^*\setminus\{0\}$.
In the case of a geodesic $\gamma:J\to\R$ on an unbounded interval $J\subset\R$, there exists a pair $(\lambda,\xi)\neq (0,0)$ for which \ref{PMP:Lie} is satisfied for every bounded subinterval $I\subset J$.

In this section we will consider properties of asymptotic cones of geodesics from the point of view of the Pontryagin maximum principle. The next lemma describes what happens to \ref{PMP:Lie} for dilations of geodesics.

\begin{lemma}\label{lemma:dilated PMP}
	Let $\gamma:I\to G$ be a horizontal curve in $G$ that satisfies \ref{PMP:Lie} for a pair $(\lambda,\xi)$. Then for any $h>0$, the dilated curve $\gamma_h:I_h\to G$ satisfies \ref{PMP:Lie} for the pair $(\dil{h}^*\lambda,h\xi)$.
\end{lemma}
\begin{proof}
	We may suppose without loss of generality that the interval $I$ is bounded and the dilation 
	\[ \gamma_h(t) = \dil{1/h}\circ\gamma(\bar{t}+ht) \]
	is happening at $\bar{t}=0$.
	The dilations are homomorphisms, so by the definition of $\Ad_g$ as the differential of the conjugation $x\mapsto gxg^{-1}$, the map $\Ad_{\gamma(t)}$ can be written in terms of $\Ad_{\gamma_h(t)}$ as
	\[ \Ad_{\gamma(t)} = \Ad_{\dil{h}\circ\gamma_h(t/h)} =  (\dil{h})_*\circ\Ad_{\gamma_h(t/h)}\circ(\dil{1/h})_*. \]
	Therefore, \ref{PMP:Lie} for $\gamma$ gives the identity
	\begin{equation}\label{eq:dilated PMP}
	\xi\innerproduct{u_\gamma}{v}
	= \lambda\left(\int_{I} \Ad_{\gamma(t)}v(t)\,dt \right)
	= (\dil{h}^*\lambda)\left(\int_{I} \Ad_{\gamma_h(t/h)}\frac{1}{h}v(t)\,dt \right).
	\end{equation}
	Denote for each $v\in L^2(I;V_1)$ by $\tilde{v}\in L^2(I_h;V_1)$ the reparametrized function $\tilde{v}(t) = v(ht)$. Then after a change of variables, the right hand side of \eqref{eq:dilated PMP} is
	\begin{equation}\label{dilpmp:change of vars 1}
		\int_{I} \Ad_{\gamma_h(t/h)}\frac{1}{h}v(t)\,dt = \int_{I_h} \Ad_{\gamma_h(t)}\tilde{v}(t)\,dt.
	\end{equation}
	Since the control $u_h$ of the dilated curve $\gamma_h$ is
	\[ u_h(t) = (\dil{1/h})_*u_\gamma(ht)\cdot h = u_\gamma(ht), \]
	a similar change of variables as in \eqref{dilpmp:change of vars 1} shows that
	\begin{equation}\label{dilpmp:change of vars 2}
		\innerproduct{u_\gamma}{v} = \int_I u_\gamma(t)v(t)\,dt = \int_{I}u_h(t/h)\tilde{v}(t/h)\,dt = h\int_{I_h}u_h(t)\tilde{v}(t)\,dt = h\innerproduct{u_h}{\tilde{v}}.
	\end{equation}
	Applying both changes of variables \eqref{dilpmp:change of vars 1} and \eqref{dilpmp:change of vars 2} to \eqref{eq:dilated PMP} gives the identity
	\[ h\xi\innerproduct{u_h}{\tilde{v}} = (\dil{h}^*\lambda)\left(\int_{I_h} \Ad_{\gamma_h(t)}\tilde{v}(t)\,dt\right). \]
	Since every element of $L^2(I_h;V_1)$ can be written as $\tilde{v}$ for some $v\in L^2(I;V_1)$, the above shows that $\gamma_h$ satifies \ref{PMP:Lie} for the pair $(\dil{h}^*\lambda,h\xi)$.
\end{proof}

In every \subfinsler Carnot group horizontal lines through the identity are infinite geodesics that are dilation invariant. Hence, the unique blowdown of any horizontal line is the line itself translated to the identity, which may or may not be abnormal. For all other curves however, every blowdown is necessarily an abnormal curve:
\begin{proposition}\label{prop:abnormal blowdowns}
	In \subriemannian Carnot groups asymptotic cones of non-line infinite geodesics are abnormal curves.
\end{proposition}
\begin{proof} The argument is partially inspired by \cite{Agrachev:Torino}.
Let $\gamma$ be a geodesic in $G$ and let $(\lambda,\xi)\in\g^*\times\R$ be a pair for which $\gamma$ satisfies \ref{PMP:Lie}. We decompose $\lambda$ as $\lambda=\lambda^{(1)}+\dots+ \lambda^{(s)}\in V_1^*\oplus\dots\oplus V_s^*\simeq \g^*$ and let $j\in\{1,\dots,s\}$ be the largest index for which $\lambda^{(j)}\neq 0$.

If $\lambda^{(2)}=\dots=\lambda^{(s)}=0$, then \ref{PMP:Lie} reduces to
\[ \lambda\Big(\int_I v\Big)=\xi\innerproduct{u_\gamma}{v}\quad \forall v\in L^2(I;V_1) \]
on every finite interval $I\subset\R$. Thus if $\lambda^{(2)}=\dots=\lambda^{(s)}=0$, then $u_\gamma$ is constant and $\gamma$ is a line. Assume from now on that $\gamma$ is not a line, so $j\geq 2$.

By Lemma~\ref{lemma:dilated PMP} the dilated curve $\gamma_h$ satisfies \ref{PMP:Lie} for the pair $(\dil{h}^*\lambda,h\xi)$. In terms of the decomposition into layers, we have
\[ \dil{h}^*\lambda = \dil{h}^*(\lambda^{(1)}+\dots+ \lambda^{(j)}) =h\lambda^{(1)}+\dots +h^{j}\lambda^{(j)}. \]
Note that \ref{PMP:Lie} is scale invariant with respect to the covector pair. Therefore scaling by $\frac{1}{h^j}$, we see that $\gamma_h$ satisfies \ref{PMP:Lie} also for the pair $\frac{1}{h^j}(\dil{h}^*\lambda,h\xi)=(\frac{1}{h^j}\dil{h}^*\lambda,\frac{1}{h^{j-1}}\xi)$. 
These pairs form a convergent sequence as $h\to\infty$:
\[ \lim\limits_{h\to\infty}(\frac{1}{h^j}\dil{h}^*\lambda,\frac{1}{h^{j-1}}\xi) = (\lambda_\infty,0), \]
where
\[ \lambda_\infty:=  \lim\limits_{h\to\infty}(h^{1-j}\lambda^{(1)}+h^{2-j}\lambda^{(2)}+\dots+ \lambda^{(j)}) = \lambda^{(j)}\neq 0. \]
Let $\sigma\in \Asymp(\gamma)$, so there exists some sequence $h_j\to\infty$ for which $\sigma=\lim\limits_{j\to\infty}\gamma_{h_j}$. By continuity, it follows that $\sigma$ satisfies \ref{PMP:Lie} for the pair $(\lambda_\infty,0)$, so $\sigma$ is an abnormal curve.
\end{proof}

\subsection{Rigidity of geodesics}
Proposition~\ref{prop:abnormal blowdowns} leads to the question of what abnormal infinite geodesics exist. Any curve in a Carnot group $G$ is abnormal in a product $G\times H$, where $H$ is any non-abelian Carnot group, so any infinite geodesic could be artificially made abnormal. Examples in non-product spaces can also be found by lifting such abnormals. However we are not aware of any examples of infinite abnormal geodesics in rank 2 Carnot groups, or even of \emph{strictly} abnormal infinite geodesics in Carnot groups of any rank. 

One issue in trying to find infinite (strictly) abnormal geodesics is that it is difficult to prove that any extremal is minimizing for all time. Currently known examples of non-line infinite geodesics such as the one studied in Section~\ref{sec:Engel} arise from a complete optimal synthesis. This is however often unfeasible, as the Hamiltonian system for the normal equations is often non-integrable in higher dimensional problems, see for instance the non-integrability results in \cite{BBKM:2016:nonintegrability-rank3-step3,LokutsievskiiSachkov:2018:liouville-step4}. 

In \subriemannian Carnot groups of step 2 the only infinite geodesics are lines \cite[Proposition~2.2]{Kishimoto}. The same rigidity result can be extended to all \subfinsler Carnot groups with a strictly convex norm \cite{Hakavuori-2020-step2_geodesics}. That is, if a strictly abnormal infinite geodesic exists, it must be found in step 3 or above.

\section{Applications of the existence of a line tangent}\label{sec:exists line tangent}
We next provide some consequences of the existence of a line tangent. 
\subsection{Loss of optimality}\label{sec:LOSS}
We prove that there are geodesics that lose optimality whenever they are extended.

\begin{proposition}\label{Prop:lost:min}
In every non-Abelian \subfinsler Carnot group defined by a strictly convex norm (e.g., in every \subriemannian Carnot group) there exist finite-length geodesics that cannot be extended as geodesics.
\end{proposition}
\begin{proof}
For every such group $G$, we know that the only infinite geodesics in $G/\brkt{G}{G}$ are lines.
Therefore, by Corollary~\ref{cor:tangent iteration} every geodesic has an iterated tangent that is a line. Since iterated tangents are tangents, we have that every geodesic in $G$ has a line tangent.

Fix a nonzero element $v\in V_2$, which exists since $G$ is not Abelian. Let $\gamma:[0,T]\to G$ be a geodesic with $\gamma(0)=\idelem$ and $\gamma(T)=\exp(v)$. We claim that any such geodesic cannot be extended to a geodesic $\tilde\gamma:[-\eps,T]\to G$ such that $\tilde\gamma|_{ [0,T] } = \gamma$ for any $\eps>0$.

Let $\delta_{-1}:G\to G$ be the group homomorphism such that $(\delta_{-1})_*(v)=(-1)^j v$ for all $v\in V_j$. The map $\delta_{-1}$ is an isometry, since $(\delta_{-1})_*|_{V_1}$ is an isometry. Notice that $\delta_{-1} \circ \gamma$ is another\footnote{We learned this trick for proving non-uniqueness of geodesics in Carnot groups from \cite[Proposition~3.2]{Berestovskii_Busemann}} geodesic from $\idelem$ to $\exp(v)$.
 
Suppose that an extension $\tilde{\gamma}:[\eps,T]\to G$ of $\gamma$ existed. By the existence of a line tangent outlined in the first paragraph, we have that there exists a sequence $h_j\to 0$ such that
\[ \tilde\gamma_{h_j}= \delta_{\frac{1}{h_j}}\circ \tilde\gamma\circ \delta_{h_j}\to \sigma, \]
with $\sigma(t)=\exp(tX)$ for some $X\in V_1$. 
Replace $\gamma$ by $\delta_{-1}\circ \gamma$ in the extension $\tilde{\gamma}$, i.e., consider the concatenated curve
\[ \eta:=\tilde\gamma|_{ [-\eps,0] } * (\delta_{-1}\circ \gamma). \]
Since $\gamma$ and $\delta_{-1}\circ \gamma$ are both geodesics with the same endpoints, and $\tilde{\gamma}$ was a geodesic extension of $\gamma$, the curve $\eta$ is also a geodesic.
However, $\eta$ has a blowup at $0$ that is not injective: for $t<0$,
\[ \eta_{\eps_j}(t)= (\delta_{\frac{1}{\eps_j}}\circ \eta\circ \delta_{\eps_j})(t)=\tilde\gamma_{\eps_j}(t)\to 
\exp(tX) \]
whereas for $t>0$,
\[ \eta_{\eps_j}(t)= (\delta_{\frac{1}{\eps_j}}\circ\delta_{-1}\circ \gamma\circ \delta_{\eps_j})(t)=
\delta_{-1}\circ\tilde \gamma_{\eps_j}(t)\to 
\delta_{-1} \exp(tX)= \exp(-tX). \]
Any blowup of the geodesic $\eta$ would have to be a geodesic, but this blowup is not even injective, so we get a contradiction.
\end{proof}

\subsection{Non-minimality of corners for distributions of non-constant rank}
In the previous work \cite{Hakavuori_LeDonne_2016} we proved that corners cannot be length-minimizing in any \subriemannian manifold, in which by standard definition the distribution has constant rank. We will next show that the existence of a line tangent implies that the assumption that the distribution has constant rank can be omitted.

\begin{proposition}\label{prop:non-constant rank corners}
	Let $M$ be a generalized \subriemannian manifold with a distribution not necessarily of constant rank. If $\gamma:I\to M$ is a geodesic, then $\gamma$ cannot have a corner-type singularity.
\end{proposition}
\begin{proof}
	Reparametrizing and translating, it suffices to consider the case when $I=(-\epsilon,\epsilon)$ and show that $\gamma$ cannot have a corner-type singularity at 0.
	
	Let $\tilde{M}$ be a desingularization of $M$ as in \cite[Lemma~2.5]{jeancontrol}, that is an equiregular \subriemannian manifold equipped with a canonical projection $\pi:\tilde{M}\to M$. Since the projection is a submetry, the geodesic $\gamma$ can be lifted to a geodesic $\tilde{\gamma}:I\to\tilde{M}$. Let $u:(-\epsilon,\epsilon)\to \R^r$ be the control of $\tilde{\gamma}$ with respect to a fixed frame $\tilde{X}_1,\dots,\tilde{X}_r$ of horizontal vector fields on $\tilde{M}$. Then $u$ is also the control of $\gamma$ with respect to the projected horizontal frame $\pi_*\tilde{X}_1,\dots,\pi_*\tilde{X}_r$ on $M$.
	
	By Theorem~\ref{thm:sr tangent iteration}, the curve $\tilde{\gamma}$ has an iterated tangent  that is a line, and thus also a tangent that is a line. By \cite[Remark 3.12]{Monti_Pigati_Vittone:tangent_cones} it follows that there exist a constant $v\in\R^r$ and a sequence of scales $h_j\to 0$ such that for the rescaled controls $u^{(j)}:(-\epsilon/h_j,\epsilon/h_j)\to\R^r$, $u^{(j)}(t)=u(h_jt)$, we have $u^{(j)}\to v$ in $L^2_{\text{loc}}(\R;\R^r)$.
	
	On the other hand, in coordinates near $\gamma(0)$ on $M$, for any small enough $h_j$ we have
	\begin{align*}
	\frac{\gamma(h_j)}{h_j} &= \int_{0}^{h_j}\sum_{k=1}^ru_k(t)X_k(\gamma(t))\,\frac{dt}{h_j}
	= \int_{0}^1\sum_{k=1}^ru_k^{(j)}(t)X_k(\gamma(h_jt))\,dt\quad\text{and}\\
	\frac{\gamma(-h_j)}{-h_j} &= \int_{-h_j}^0\sum_{k=1}^ru_k(t)X_k(\gamma(t))\,\frac{dt}{h_j}
	= \int_{-1}^0\sum_{k=1}^ru_k^{(j)}(t)X_k(\gamma(h_jt))\,dt.
	\end{align*}
	By continuity of the vector fields $X_1,\dots,X_r$ and the convergence $u^{(j)}\to v$ in $L^2_{\text{loc}}$, we see that
	\begin{equation*}
	\lim\limits_{j\to\infty}\frac{\gamma(h_j)}{h_j} = \sum_{k=1}^{r}v_kX_k(\gamma(0)) = \lim\limits_{j\to\infty}\frac{\gamma(-h_j)}{-h_j}.
	\end{equation*}
	In particular, if $\gamma$ has one-sided derivatives at $0$, then they must be equal, so $\gamma$ cannot have a corner-type singularity.
\end{proof}

\section{On sharpness of Theorem~\ref{thm:geodesic blowdowns}}\label{sec:Sharp}
In this section we want to consider whether Theorem~\ref{thm:geodesic blowdowns} can be improved. In particular, we will show that in the statement of the theorem, taking the horizontal projection is essential. That is, there exist geodesics that are not in a finite neighborhood of any proper Carnot subgroup (see Corollary~\ref{cor:geodesic not in a line nbhd}). 

A possible improvement of Theorem~\ref{thm:geodesic blowdowns} would be to strengthen the claim in the horizontal projection. Namely, the following might be true.

\begin{conjecture}\label{conj:geodesic blowdowns} 
	If $\gamma:\mathbb{R}\to G$ is a geodesic such that $\horproj\circ\gamma:\mathbb{R}\to G/\brkt{G}{G}$ is not a geodesic, then there exists a hyperplane $W\subset V_1$ such that $\lim\limits_{t\to\pm\infty}d(\horproj\circ\gamma(t),W) = 0$.
\end{conjecture}

Toward this conjecture, we shall consider the case of rank 2 Carnot groups, where proper Carnot subgroups are simply lines. 
For this reason, in the next subsection we first prove some general statements about lines that are a finite distance apart.

\subsection{Lines in Carnot groups}\label{sec:Lines}
A {\em  line} in a Lie group is a left-translation of a one-parameter subgroup, i.e., a curve $L:\R\to G$ such that $L(t)=g\exp(tX)$ for some $g\in G$ and $X\in\mathfrak{g}$. We stress that  in case $G$ is a Carnot group the vector  $X$ is not assumed to be  horizontal.

The distance between lines will be measured by the Hausdorff distance:
The Hausdorff distance of two subsets $A,B\subset G$ is 
\[ \dhaus(A,B) := \max\Big(\sup_{a\in A}d(a,B),\sup_{b\in B}d(b,A)\Big). \]
In Lemma~\ref{lemma:lines at a finite distance} we will give two equivalent algebraic conditions for two lines to be at a bounded distance from each other. In the proof we will want to use also the notion of distance of lines given by the sup-norm, which is parametrization dependent. For this reason we first prove a sufficient condition (Lemma~\ref{lemma:finite hausdorff distance}) for the equivalence of boundedness of Hausdorff distance and boundedness of sup-norm.
This result is naturally stated in much more generality than just lines in Carnot groups.

\begin{lemma}\label{lemma:finite hausdorff distance}
	Let $X$ and $Y$ be metric spaces, and let $\alpha:X\to Y$ and $\beta:X\to Y$ be maps such that the following conditions hold.
	\begin{enumerate}[label=(\alph{*})]
		\item\label{lenum:finite hausdorff:beta bounded} The map $\beta$ is bornologous: For every $R<\infty$, there exists $R'<\infty$ such that $\beta(B_X(x,R))\subset B_Y(\beta(x),R')$ for any $x\in X$.
		\item\label{lenum:finite hausdorff:combined expansive} The map $\alpha\times\beta:X^2\to Y^2$ maps distant points to distant points: For every $M<\infty$, there exists $R<\infty$ such that $d(\alpha(x_1),\beta(x_2))> M$ for any $x_1,x_2$ with $d_X(x_1,x_2)>R$.
	\end{enumerate}
	Then $\dhaus(\alpha(X),\beta(X)) < \infty$ if and only if $\sup_{x\in X}d(\alpha(x),\beta(x)) < \infty$.
\end{lemma}
\begin{proof}
	Clearly $\dhaus(\alpha(X),\beta(X)) \leq \sup_{x\in X}d(\alpha(x),\beta(x))$ so it suffices to prove the ``only if'' implication. That is, we assume that $M:=\dhaus(\alpha(X),\beta(X))<\infty$ and we will show that also $\sup_{x\in X}d(\alpha(x),\beta(x)) < \infty$.
	
	By the definition of the Hausdorff distance, we have $d(\alpha(x),\beta(X))\leq M$ for every $x\in X$. Therefore there exists a (possibly discontinuous) map $f:X\to X$ choosing roughly closest points from $\beta(X)$, i.e., a map such that
	\begin{equation}\label{eq:closest point map}
		d(\alpha(x),\beta\circ f(x)) \leq M+1\quad\forall x\in X.
	\end{equation}
	Let $R<\infty$ be the constant given by the assumption \ref{lenum:finite hausdorff:combined expansive} such that $d(\alpha(x_1),\beta(x_2))>M+1$ for any $x_1,x_2\in X$ with $d(x_1,x_2)>R$. Then the bound \eqref{eq:closest point map} implies that 
	\begin{equation*}
		d(x,f(x)) \leq R\quad \forall x\in X.
	\end{equation*}
	Assumption \ref{lenum:finite hausdorff:beta bounded} then implies that there exists $R'<\infty$ such that
	\begin{equation}\label{eq:beta offset bound}
		d(\beta(x),\beta\circ f(x))\leq R'\quad\forall x\in X.
	\end{equation}
	Combining the bounds \eqref{eq:closest point map} and \eqref{eq:beta offset bound}, we get for any $x\in X$ the uniform bound
	\[ d(\alpha(x),\beta(x))\leq d(\alpha(x),\beta\circ f(x))+d(\beta\circ f(x),\beta(x))\leq M+1+R'<\infty, \]
	proving the claim.
\end{proof}

\begin{lemma}\label{lemma:lines at a finite distance}
Assume $G$ is a  Carnot group and let $L_1(t)=g\exp(tX)$ and $L_2(t)=h\exp(tY)$ be two lines in the group. The following are equivalent:
\begin{enumerate}[label=(\roman{*})]
	\item\label{lenum:lines:direction vectors} There exists a constant $c>0$ such that $X=c\Ad_{g^{-1}h}Y$.
	\item\label{lenum:lines:right translation} There exist a constant $c>0$ and an element $k\in G$ such that $L_1(t) = L_2(ct)k$.
	\item\label{lenum:lines:hausdorff distance} $\dhaus(L_1(\R_+),L_2(\R_+))<\infty$.
\end{enumerate}
\end{lemma}
\begin{proof}
	The equivalence of \ref{lenum:lines:direction vectors} and \ref{lenum:lines:right translation} is an algebraic computation: For any $k\in G$ and $Z\in\mathfrak{g}$, we have the identity
	\[ k\exp(Z) = k\exp(Z)k^{-1}k = C_{k}(\exp(Z))\cdot k = \exp(\Ad_kZ)k. \]
	For any $c>0$, we apply the above with $k=g^{-1}h$ and $Z=ctY$. This gives the identity
	\begin{align}\label{eq:adjoint vs difference}
	L_1(t)^{-1}L_2(ct)
	&=(g\exp(tX))^{-1}\cdot (h\exp(ctY))
	\\\nonumber&= \exp(-tX)\exp(ct\Ad_{g^{-1}h}Y)g^{-1}h.
	\end{align}
	If \ref{lenum:lines:direction vectors} holds, then \eqref{eq:adjoint vs difference} implies that $L_1(t)^{-1}L_2(ct)$ is constant, proving \ref{lenum:lines:right translation}. Vice versa, if  \ref{lenum:lines:right translation} holds, then $L_1(t)^{-1}L_2(ct)$ is constant, so \eqref{eq:adjoint vs difference} is constant. But this is only possible if \ref{lenum:lines:direction vectors} holds.
	
	That \ref{lenum:lines:right translation} implies \ref{lenum:lines:hausdorff distance} is immediate from the left-invariance of the distance on $G$. It remains to prove that \ref{lenum:lines:hausdorff distance} implies \ref{lenum:lines:right translation}.
	The claim is equivalent to saying that the product $L_1(t)^{-1}L_2(ct)$ is constant for some $c>0$. Since the product is in exponential coordinates a polynomial expression, it suffices to show that 
	\begin{equation}\label{eq:line sup dist finite}
		\sup_{t\in\R_+}d(L_1(t),L_2(ct))<\infty.
	\end{equation}
	We will prove this by induction on the step of the group $G$. In a normed space, two half-lines are a finite distance apart if and only if they are parallel, so the claim holds in step 1.
	
	Suppose that the claim is true for all Carnot groups of step at most $s-1$ and suppose that $G$ is of step $s$. 
	We will prove \eqref{eq:line sup dist finite} by applying Lemma~\ref{lemma:finite hausdorff distance} to the curves $\alpha(t)=L_2(ct)$ and $\beta(t)=L_1(t)$ for some $c>0$ to be fixed later. 
	
	From the identity
	\[ d(L_1(t_2),L_1(t_1)) = d(L_1(0),L_1(t_1-t_2)) \]
	we see that $R' = \sup_{\abs{t}\leq R}d(L_1(0),L_1(t))<\infty$ satisfies assumption \ref{lenum:finite hausdorff:beta bounded} of Lemma~\ref{lemma:finite hausdorff distance}.
	
	For assumption \ref{lenum:finite hausdorff:combined expansive} of  Lemma~\ref{lemma:finite hausdorff distance}, we need a lower bound for the value $d(L_1(t_1),L_2(ct_2))$. We consider first the case when the lines degenerate under the projection $\stepproj:G\to G/\exp(V_s)$ to step $s-1$, i.e., when $X,Y\in V_s$. 
	Since elements in $\exp(V_s)$ commute with everything, for any $t_1,t_2\in\R_+$ we have that 
	\begin{equation*}
	d(L_1(t_1),L_2(t_2)) = d(\idelem, g^{-1}h\exp(t_2Y-t_1X)).
	\end{equation*}
	If $Y=cX$ for some $c>0$, then condition \ref{lenum:lines:direction vectors} is satisfied, which implies \eqref{eq:line sup dist finite} by the first part of the proof. Otherwise, $t_2Y-t_1X$ escapes any compact subset of $V_s$ as $\abs{t_2-t_1}\to\infty$. Recall that in Carnot groups the exponential map is a global diffeomorphism and the distance function is proper. Hence, the lower bound
	\[ d(L_1(t_1),L_2(t_2)) \geq d(\idelem,\exp(t_2Y_2-t_1Y_1))-d(\idelem,g^{-1}h) \]
	implies that assumption \ref{lenum:finite hausdorff:combined expansive} of Lemma~\ref{lemma:finite hausdorff distance} is satisfied for any $c>0$. By Lemma~\ref{lemma:finite hausdorff distance} we conclude that in this case \eqref{eq:line sup dist finite} is satisfied for any $c>0$.
	
	Next we consider the case when at least one of the lines does not degenerate under the projection $\stepproj:G\to G/\exp(V_s)$. Since the projection $\stepproj$ is 1-Lipschitz, we have
	\[ \dhaus(\stepproj\circ L_1(\R_+),\stepproj\circ L_2(\R_+)) \leq \dhaus(L_1(\R_+),L_2(\R_+)) < \infty. \]
	Note that the above implies that also the other line cannot degenerate to a constant.
	
	By the inductive assumption in the step $s-1$ Carnot group $G/\exp(V_s)$, we can fix $c>0$ such that
	\begin{equation*}
	M:=\sup_{t\in\R_+}d(\stepproj\circ L_1(t),\stepproj\circ L_2(ct))< \infty.
	\end{equation*}
	It follows that for any $t_1,t_2\in\R_+$ we get the lower bound
	\begin{align}\label{two lines:line difference bound}
	d(L_1(t_1),L_2(ct_2))
	&\nonumber\geq d(\stepproj\circ L_1(t_1),\stepproj\circ L_2(ct_2))
	\\&\geq d(\stepproj\circ L_1(t_1),\stepproj\circ L_1(t_2)) 
	\\\nonumber&\quad- d(\stepproj\circ L_1(t_2),\stepproj\circ L_2(ct_2))
	\\\nonumber&\geq d(\stepproj\circ L_1(t_1),\stepproj\circ L_1(t_2)) -M.
	\end{align}
	Decompose the direction vector of $L_1$ into homogeneous components as $X = X_{(1)}+\dots+X_{(s)}\in V_1\oplus\dots\oplus V_s$ and let $k$ be the smallest index for which $X_{(k)}\neq 0$. Since $\stepproj\circ L_1$ is non-constant, we have $k\leq s-1$.
	By homogeneity of the distance in the projection to step $k$ we get the lower bound
	\begin{align}\label{two lines:line distance growth}
	\nonumber d(\stepproj\circ L_1(t_1),\stepproj\circ L_1(t_2))
	&\geq d(\multistepproj{k}\circ L_1(0),\multistepproj{k}\circ L_1(t_2-t_1))
	\\&=d(\multistepproj{k}(\idelem), \multistepproj{k}\circ\exp((t_2-t_1)X_{(k)}) )
	\\\nonumber &=\abs{t_2-t_1}^{1/k}d(\multistepproj{k}(\idelem), \multistepproj{k}\circ\exp(X_{(k)}) ).
	\end{align}
	Combining \eqref{two lines:line difference bound} and \eqref{two lines:line distance growth} and denoting $C:=d(\multistepproj{k}(\idelem), \multistepproj{k}\circ\exp(X_{(k)}) )>0$, we have that
	\[ d(L_1(t_1),L_2(ct_2)) \geq C\abs{t_2-t_1}^{1/k}-M. \]
	This shows that assumption \ref{lenum:finite hausdorff:combined expansive} of Lemma~\ref{lemma:finite hausdorff distance} holds for $\alpha(t)=L_2(ct)$ and $\beta(t)=L_1(t)$, so Lemma~\ref{lemma:finite hausdorff distance} implies that we have \eqref{eq:line sup dist finite}.
\end{proof}

\subsection{An explicit infinite non-line geodesic in the Engel group}\label{sec:Engel}
The \subriemannian Engel group $\engel$ is a \subriemannian Carnot group  of rank 2 and step 3 of dimension 4. Its Lie algebra $\mathfrak{g}$ has a basis $X_1,X_2,X_{12},X_{112}$ whose only non-zero commutators are
\[ \brkt{X_1}{X_2} = X_{12}\quad\text{and}\quad \brkt{X_{1}}{X_{12}}=X_{112}. \]
In \cite{Ardentov_Sachkov:engel_cut_time}, Ardentov and Sachkov studied the cut time for normal extremals in the Engel group, and found a family of infinite geodesics that are not lines. These geodesics have a property stronger than that implied by Theorem~\ref{thm:geodesic blowdowns}. Namely, instead of merely having their horizontal projections contained in a finite neighborhood of a hyperplane, their horizontal projections are in fact asymptotic to a line.

To study these infinite geodesics explicitly, we will consider exponential coordinates
\[ \R^4\to \engel,\quad x=(x_1,x_2,x_{12},x_{112})\mapsto \exp(x_1X_1+x_2X_2+x_{12}X_{12}+x_{112}X_{112}). \]
By the BCH formula, the group law is given by $x\cdot y = z$, where
\begin{align}\label{Engel:group law}
	\nonumber z_1 &= x_{1} + y_{1}\\
	z_2 &= x_{2} + y_{2}\\
	\nonumber z_{12} &= x_{12} + y_{12} + \frac{1}{2}(x_1y_2 - x_2y_1)\\
	\nonumber z_{112} &= x_{112} + y_{112} 
	+\frac{1}{2}(x_{1} y_{12} - x_{12} y_{1}) 
	+\frac{1}{12}(x_{1}^{2} y_{2} - x_{1} x_{2} y_{1} - x_{1} y_{1} y_{2} + x_{2} y_{1}^{2}).
\end{align}
The left-invariant extensions of the horizontal basis vectors $X_1,X_{2}$ are
\begin{align}\label{Engel:left invariant frame}
X_1(x) &= \partial_{1} -\frac{1}{2} x_{2}\partial_{12} -(\frac{1}{12} x_{1} x_{2} + \frac{1}{2} x_{12})\partial_{112}\quad\text{and}\\
\nonumber X_2(x) &= \partial_{2} + \frac{1}{2} x_{1}\partial_{12} + \frac{1}{12} x_{1}^{2}\partial_{112}.
\end{align}
Note that the coordinates used in \cite{Ardentov_Sachkov:engel_cut_time} are not exponential coordinates, but the two coordinate systems agree in the horizontal ($x_1$ and $x_2$) components.

Given a covector written in the dual basis as $\lambda=(\lambda_1,\lambda_2,\lambda_{12},\lambda_{112})\in\mathfrak{g}^*$, the normal equation given by \ref{PMP:Lie} takes the form
\begin{equation}\label{normal equation:Engel}
	u_\gamma(t) = \lambda\left( \Ad_{\gamma(t)}X_1 \right)X_1+\lambda\left( \Ad_{\gamma(t)}X_2 \right)X_2.
\end{equation}
In \cite{Ardentov_Sachkov:engel_cut_time}, the space of covectors $\mathfrak{g}^*$ is stratified into 7 different classes $C_1,\dots,C_7$ based on the different types of trajectories of the corresponding normal extremals. For our purposes the relevant class is $C_3$, which consists of the non-line infinite geodesics. In \cite{Ardentov_Sachkov:engel_cut_time}, the class was parametrized by 
\[ C_3 = \{(\cos(\theta+\pi/2),\sin(\theta+\pi/2),c,\alpha):\; \alpha\neq 0, \frac{c^2}{2}-\alpha\cos\theta=\abs{\alpha}, c\neq 0 \}.  \]
An example of a covector $\lambda\in\mathfrak{g}^*$ in the class $C_3$ is $\lambda = (0,1,2,1)$. However, instead of integrating the normal equation \eqref{normal equation:Engel} with this covector, we will consider a translation of the curve to simplify the asymptotic study of the resulting curve. Instead of considering the geodesic starting from $(0,0,0,0)$, we will consider the translated geodesic starting from $(2,0,0,0)$. 

If $\gamma:\R\to\engel$ satisfies \eqref{normal equation:Engel} with the covector $\lambda$, then a left-translation $\beta=g\gamma:\R\to\engel$ by $g\in\engel$ satisfies
\begin{align}\label{Engel:translated normal equation}
	u_{\beta}(t) &= \lambda\left( \Ad_{\gamma(t)}X_1 \right)X_1+\lambda\left( \Ad_{\gamma(t)}X_2 \right)X_2
	\\\nonumber&=\lambda\left( \Ad_{g^{-1}\beta(t)}X_1 \right)X_1+\lambda\left( \Ad_{g^{-1}\beta(t)}X_2 \right)X_2.
\end{align}
Using the formula $\Ad_{\exp(Y)}X = e^{\ad(Y)}X$, we compute for $x=(x_1,x_2,x_{12},x_{112})\in\engel$ that
\begin{align*}
\Ad_{x}X_1 &= X_1 - x_2X_{12} - (x_{12} + \frac{1}{2}x_1x_2)X_{112}\quad\text{and}\\
\nonumber\Ad_{x}X_2 &= X_2 + x_1X_{12} + \frac{1}{2}x_1^2X_{112}.
\end{align*}
Evaluated for the covector $\lambda=(0,1,2,1)$, we get
\begin{align}\label{Engel:infinite geodesic covector adjoints}
\lambda(\Ad_xX_1) &= -2 x_{2} - x_{12} - \frac{1}{2} x_{1} x_{2},\\
\nonumber\lambda(\Ad_xX_2) &= 1 + 2x_1 + \frac{1}{2} x_{1}^{2}.
\end{align}
By the group law \eqref{Engel:group law}, the translation of the curve in which we are interested is
\[ (2,0,0,0)^{-1} \cdot \beta = \left(\beta_{1} - 2,\,\beta_{2},\,\beta_{12} - \beta_{2},\,\beta_{112} - \beta_{12} + \frac{1}{3} \beta_{2} + \frac{1}{6} \beta_{1} \beta_{2}\right) .\]
Substituting the points $x=(2,0,0,0)^{-1} \cdot \beta(t)$ into \eqref{Engel:translated normal equation} using \eqref{Engel:infinite geodesic covector adjoints}, we get the ODE
\begin{align}\label{Engel:infinite geodesic ODE}
\dot{\beta}_1 &
=-\frac{1}{2} \beta_{1} \beta_{2} - \beta_{12}
&
\dot{\beta}_2 &
=\frac{1}{2} \beta_{1}^{2} - 1.
\end{align}
\begin{figure}
	\begin{center}
		\includegraphics{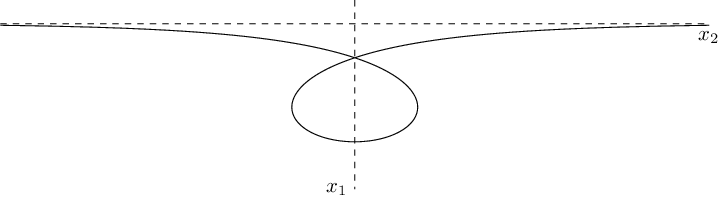}
	\end{center}
	\caption{Horizontal projection of the non-line Engel infinite geodesic $\beta$ (rotated $90^\circ$ clockwise)}\label{figure:engel geodesic projection}
\end{figure}

\begin{lemma}\label{lemma:Engel explicit geodesic}
	The horizontal curve $\beta:\R\to \engel$ satisfying the ODE \eqref{Engel:infinite geodesic ODE} with the initial condition $\beta(0)=(2,0,0,0)$ has the explicit form (see Figure~\ref{figure:engel geodesic projection})
	\begin{align*}
	\beta_1(t)&=\frac{2}{\cosh(t)},&
	\beta_2(t)&=2 \, \tanh(t)-t,\\
	\beta_{12}(t)&=\frac{t}{\cosh(t)},&
	\beta_{112}(t)&=\frac{2}{3}\tanh(t) - \frac{t}{3 \cosh(t)^{2}}.
	\end{align*}
\end{lemma}
\begin{proof}
	The proof of the lemma is a direct computation. First we shall verify horizontality of $\beta$, i.e., that $\dot{\beta}(t) = \dot{\beta}_1(t)X_1(\beta(t))+\dot{\beta}_2(t)X_2(\beta(t))$. By the coordinate form \eqref{Engel:left invariant frame} of the left-invariant frame, we need to check that
	\begin{align}\label{eq:horizontality condition1}
	\dot{\beta}_{12} &= \frac{1}{2}(\beta_1\dot{\beta}_2 - \beta_2\dot{\beta}_1)\quad\text{and}\\\label{eq:horizontality condition2}
	\dot{\beta}_{112} &= \frac{1}{12} \beta_{1}^{2} \dot{\beta}_{2} - \Big(\frac{1}{12} \beta_{1} \beta_{2} + \frac{1}{2} \beta_{12} \Big)\dot{\beta}_{1}.
	\end{align}
	From the given explicit form of $\beta$, we compute the derivatives
	\begin{align*}
	\dot{\beta}_1&=-\frac{2 \, \sinh(t)}{\cosh(t)^{2}},\\
	\dot{\beta}_2&=2(1-\tanh(t)^2)-1 
	= 1-2 \, \tanh(t)^{2},\\
	\dot{\beta}_{12}&=\frac{\cosh(t)-t \sinh(t)}{\cosh(t)^{2}},\\
	\dot{\beta}_{112}
	&=\frac{2}{3\cosh(t)^2}-\frac{\cosh(t)-2t\sinh(t)}{3\cosh(t)^3}
	=\frac{\cosh(t) + 2 t \sinh(t)}{3 \cosh(t)^{3}}.
	\end{align*}
	Expanding the right hand side $\frac{1}{2}(\beta_1\dot{\beta}_2 - \beta_2\dot{\beta}_1)$ of \eqref{eq:horizontality condition1}, we get
	\begin{align*}
	&\frac{1}{2}\left(\frac{2}{\cosh(t)}\Big(1- \frac{2\sinh(t)^2}{\cosh(t)^2}\Big) - \Big(\frac{2\sinh(t)}{\cosh(t)}-t\Big)\Big(-\frac{2\sinh(t)}{\cosh(t)^2}\Big)\right)
	\\&=\frac{1}{\cosh(t)}-\frac{2\sinh(t)^2}{\cosh(t)^3}+\frac{2\sinh(t)^2}{\cosh(t)^3}-\frac{t\sinh(t)}{\cosh(t)^2}
	\\&=\frac{\cosh(t)-2t\sinh(t)}{\cosh(t)^2},
	\end{align*}
	which is exactly $\dot{\beta}_{12}$. Similarly, expanding the right hand side $\frac{1}{12} \beta_{1}^{2} \dot{\beta}_{2} - \Big(\frac{1}{12} \beta_{1} \beta_{2} + \frac{1}{2} \beta_{12} \Big)\dot{\beta}_{1}$ of \eqref{eq:horizontality condition2}, we get
	\begin{align*}
	&\frac{1}{12}\frac{4}{\cosh(t)^2}\Big(1-\frac{2\sinh(t)^2}{\cosh(t)^2}\Big)
	\\&-\Big(
	\frac{1}{12}\frac{2}{\cosh(t)}\Big(\frac{2\sinh(t)}{\cosh(t)}-t\Big)
	+\frac{1}{2}\frac{t}{\cosh(t)}\Big)
	\Big)
	\Big(-\frac{2\sinh(t)}{\cosh(t)^2}\Big)
	\\&=\frac{1}{3\cosh(t)^2}-\frac{2\sinh(t)^2}{3\cosh(t)^4}+\frac{2\sinh(t)^2}{3\cosh(t)^4}+\frac{2t\sinh(t)}{3\cosh(t)^3}
	\\&=\frac{\cosh(t)+2t\sinh(t)}{3\cosh(t)^3},
	\end{align*}
	which is exactly $\dot{\beta}_{112}$, proving horizontality of the curve $\beta$.
	
	Finally, we verify that $\beta$ satisfies the ODE \eqref{Engel:infinite geodesic ODE}. Once again, expanding the right hand sides we get
	\begin{align*}
	-\frac{1}{2}\beta_1\beta_2-\beta_{12}
	&=-\frac{1}{2}\frac{2}{\cosh(t)}\Big(2\frac{\sinh(t)}{\cosh(t)}-t\Big)-\frac{t}{\cosh(t)}
	=-\frac{2\sinh(t)}{\cosh(t)^2}
	=\dot{\beta}_1
	\end{align*}
	and
	\begin{align*}
	\frac{1}{2}\beta_1^2-1
	&=\frac{1}{2}\frac{4}{\cosh(t)^2}-1
	=\frac{2}{\cosh(t)^2}-1
	=1-2\tanh(t)^2=\dot{\beta}_2.\qedhere
	\end{align*}
\end{proof}

From the explicit form of the infinite geodesic $\beta$ we can deduce two properties stronger than that of Theorem~\ref{thm:geodesic blowdowns}: its horizontal projection is asymptotic to a line and the curve itself is in a finite neighborhood of a line.

\begin{proposition}\label{prop:Engel geodesic asymptotes}
	Let $L:\R\to\engel$, $L(t)=\exp(-tX_2)$, which is the abnormal line in the Engel group, and let $\beta:\R\to\engel$ be the infinite geodesic of Lemma~\ref{lemma:Engel explicit geodesic}. Then
	\begin{align*}
	\lim\limits_{t\to\infty}d(\beta(t),\exp(\tfrac{2}{3}X_{112})L(t-2)) &= 0\quad\text{and}\\
	\lim\limits_{t\to-\infty}d(\beta(t),\exp(-\tfrac{2}{3}X_{112})L(t+2)) &= 0.
	\end{align*}
\end{proposition}
\begin{proof}
	To prove the claim, we will consider the distances $$d(\exp(bX_{112})\exp(-(t+a)X_2),\beta(t)),$$ where $a,b\in\R$ are some constants.
	This distance is zero exactly when the product
	\[ z(t) = (0,t+a,0,-b)\cdot \beta(t) \] 
	vanishes.
	
	By the group law \eqref{Engel:group law} and the explicit form of the components given in Lemma~\ref{lemma:Engel explicit geodesic}, we see that the components of the product $z(t)$ are
	\begin{align*}
	z_1(t) &= \beta_1(t) = \frac{2}{\cosh(t)},
	\\
	z_2(t) &= \beta_2(t) + t + a = 2\tanh(t) + a,
	\\
	z_{12}(t) &= \beta_{12}(t) -\frac{1}{2}(t+a)\beta_1(t)
	= - \frac{a}{\cosh(t)}\quad\text{and}
	\\
	z_{112}(t) &= \beta_{112}(t) +\frac{1}{12}(t+a)\beta_1(t)^2- b
	=\frac{2}{3}\tanh(t)+ \frac{a}{3\cosh(t)^2} - b.
	\end{align*}
	From the above we deduce that 
	\[ \lim_{t\to\infty}z(t) = (0,2+a,0,2/3-b)\quad\text{and}\quad \lim_{t\to-\infty}z(t)=(0,-2+a,0,-2/3-b). \]
	and the claim of the proposition follows.
\end{proof}
\begin{corollary}\label{cor:Engel geodesic asymptotic properties}
	Let $L:\R\to\engel$, $L(t)=\exp(-tX_2)$, which is the abnormal line in the Engel group, and let $\beta:\R\to\engel$ be the infinite geodesic of Lemma~\ref{lemma:Engel explicit geodesic}. Then
	\[ \lim\limits_{t\to\pm\infty}d(\horproj\circ\beta(t),\horproj\circ L(t)) = 0\quad\text{and}\quad \sup_{t\in\R}d(\beta(t),L) < \infty. \]
\end{corollary}
\begin{proof}
	Since the horizontal projections of the elements $\exp(\pm\frac{2}{3}X_{112})$ are zero, the lines in Proposition~\ref{prop:Engel geodesic asymptotes} have the same horizontal projection as the abnormal line $L$, and the claim $\lim\limits_{t\to\pm\infty}d(\horproj\circ\beta(t),\horproj\circ L(t)) = 0$ follows.
	
	On the other hand, the elements $\exp(\pm\frac{2}{3}X_{112})$ are also in the center of the Engel group, so for all $t\in\R$ we have
	\begin{align*}
	d(L(t),\exp(\tfrac{2}{3}X_{112})L(t-2)) &\leq  d(L(t),L(t-2))
	\\&\quad+d(L(t-2),\exp(\tfrac{2}{3}X_{112})L(t-2))
	\\&=2 + d(\idengel, \exp(\tfrac{2}{3}X_{112})).
	\end{align*}
	Thus Proposition~\ref{prop:Engel geodesic asymptotes} implies that
	\begin{align*}
	\sup_{t\in\R_+}d(\beta(t),L(t)) &\leq \sup_{t\in\R_+}d(\beta(t),\exp(\tfrac{2}{3}X_{112})L(t-2))
	\\&\quad+ d(\exp(\tfrac{2}{3}X_{112})L(t-2),L(t)) < \infty.
	\end{align*}
	Similarly using the triangle inequality through $\exp(-\frac{2}{3}X_{112})L(t+2)$ instead of $\exp(\frac{2}{3}X_{112})L(t-2)$, we see that $\sup_{t\in\R_-}d(\beta(t),L(t))<\infty$, proving the claim.
\end{proof}

\subsection{Lift of the infinite non-line geodesic to step 4}
We shall next show that Theorem~\ref{thm:geodesic blowdowns} cannot be improved to say that every \subriemannian geodesic is at a finite distance from a lower rank subgroup. Although by Corollary~\ref{cor:Engel geodesic asymptotic properties}, this stronger claim is true for the Engel group, the claim is no longer true for the lift of the geodesic $\beta$ from Lemma~\ref{lemma:Engel explicit geodesic} to a specific Carnot group of rank 2 and step 4.

We will prove the claim by showing that the mismatched limits
\[ \lim\limits_{t\to\infty}\beta_{112}(t) = \frac{2}{3}\neq -\frac{2}{3}=\lim\limits_{t\to-\infty}\beta_{112}(t) \]
will cause the lift of $\beta$ to have different lines as asymptotes as $t\to\infty$ and as $t\to -\infty$ (Proposition~\ref{prop:different asymptote lines geodesic}). The claim will then follow from Lemma~\ref{lemma:lines at a finite distance}, where we proved that the only lines a finite distance apart are right translations of one another.

The specific Carnot group $G$ where we will consider a lift of the Engel geodesic $\beta$ is the one whose Lie algebra $\mathfrak{g}$ has the basis $$X_1,X_2,X_{12},X_{112},X_{122},X_{1122},$$
whose only non-zero commutators are (see Figure~\ref{figure:step 4 algebra} for a visual description)
\begin{align*}
\brkt{X_1}{X_2} &= X_{12},&
\brkt{X_1}{X_{12}} &= X_{112},\\
\brkt{X_{12}}{X_2} &= X_{122},&
\brkt{X_1}{X_{122}} &= \brkt{X_{112}}{X_2} = X_{1122}.
\end{align*}
\begin{figure}
	\includegraphics{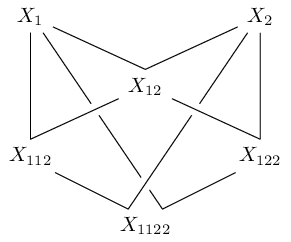}
	\caption{Diagram of relations in the step 4 Lie algebra $\mathfrak{g}$ with the Engel Lie algebra as a quotient.}\label{figure:step 4 algebra}
\end{figure}
The Lie algebra of the Engel group is a quotient of $\mathfrak{g}$ by the ideal generated by $X_{122}$, so the Engel group is the quotient of $G$ by the subgroup $H=\exp(\Span\{X_{122},X_{1122}\})$. The metric on $G$ is the \subriemannian metric such that the projection $\proj_\engel:G\to \engel=G/H$ to the Engel group is a submetry.

Let $\beta:\R\to\engel$ be the geodesic in the Engel group $\engel$ given in Lemma~\ref{lemma:Engel explicit geodesic}.
In exponential coordinates on the Engel group, $\beta(0)=(2,0,0,0)$, so for any initial point $x_0 = (2,0,0,0,x_{122},x_{1122})\in G$ there exists a horizontal lift of $\beta$ to $G$ starting from $x_0$.
Let $\alpha:\R\to G$ be the horizontal lift with the initial point $\alpha(0)=(2,0,0,0,2/3,0)$. As with $\beta_1(0)=2$, the initial coordinate $\alpha_{122}(0)=2/3$ will simplify the asymptotic behavior.
Since the projection $\proj_\engel:G\to \engel$ is a submetry and $\proj_\engel\circ\alpha=\beta$ is an infinite geodesic, the curve $\alpha$ is an infinite geodesic in $G$.

To study the lift $\alpha$, we will work in exponential coordinates. The group law is once again given by the BCH formula, which in a nilpotent Lie algebra of step 4 takes the form (for computation of the coefficients, see e.g. \cite[2.15]{Varadarajan:1984:Lie_groups})
\begin{align*}
\log(\exp(X)\exp(Y)) &= X + Y + \frac{1}{2}\brkt{X}{Y} + \frac{1}{12}(\brkt{X}{\brkt{X}{Y}}+\brkt{\brkt{X}{Y}}{Y})
\\&\quad+ \frac{1}{24}\brkt{X}{\brkt{\brkt{X}{Y}}{Y}}.
\end{align*}
In the first four coordinates, the group law $z=x\cdot y$ is the same as in the Engel group, so the components $z_1,z_2,z_{12},z_{112}$ are given by \eqref{Engel:group law}. In the last two coordinates, we have
\begin{align}\label{R2S4:group law}
z_{122} &= 
x_{122} + y_{122}
+ \frac{1}{2} (x_{12} y_{2} - x_{2} y_{12})
\\\nonumber&+ \frac{1}{12} (x_{1} y_{2}^{2} - x_{1} x_{2} y_{2} - x_{2} y_{1} y_{2}  + x_{2}^{2} y_{1}),\\
\nonumber z_{1122} &= \begin{aligned}[t]
x_{1122} &+ y_{1122}
+ \frac{1}{2} (x_{1} y_{122} - x_{122} y_{1} - x_{2} y_{112} + x_{112} y_{2})
\\&- \frac{1}{6} (x_{1} x_{2} y_{12}  + x_{12} y_{1} y_{2})
\\&+ \frac{1}{12} (x_{1} x_{12} y_{2} + x_{1} y_{2} y_{12} + x_{2} x_{12} y_{1} + x_{2} y_{1} y_{12})
\\&+ \frac{1}{24} (x_{1}^{2} y_{2}^{2} - x_{2}^{2} y_{1}^{2}).
\end{aligned}
\end{align}
The left-invariant extensions of the horizontal vectors $X_1$ and $X_2$ are 
\begin{align}\label{R2S4:horizontal frame}
X_1(x) &= \partial_{1} 
- \frac{1}{2}x_{2}\partial_{12} 
- (\frac{1}{12}x_{1} x_{2} + \frac{1}{2}x_{12})\partial_{112} 
+ \frac{1}{12}x_{2}^{2}\partial_{122} 
\\\nonumber&+ (\frac{1}{12}x_{12} x_{2} - \frac{1}{2}x_{122})\partial_{1122},
\\\nonumber X_2(x) &= \partial_{2} 
+ \frac{1}{2}x_{1}\partial_{12} 
+ \frac{1}{12}x_{1}^{2}\partial_{112} 
- (\frac{1}{12}x_{1} x_{2} - \frac{1}{2}x_{12})\partial_{122} 
\\\nonumber&+ (\frac{1}{12}x_{1} x_{12} + \frac{1}{2}x_{112})\partial_{1122}.
\end{align}

\begin{lemma}\label{lemma:R2S3:missing coordinate}
	In exponential coordinates, the second coordinate of degree 3 of $\alpha:\R\to G$ is
	\begin{equation*}
	\alpha_{122}(t) = \frac{t^{2} + 4}{6\cosh(t)} + \frac{t \sinh(t)}{3\cosh(t)^{2}}.
	\end{equation*}
\end{lemma}
\begin{proof}
	By the explicit form of the left-invariant frame given in \eqref{R2S4:horizontal frame}, we need to show that the given expression for $\alpha_{122}$ satisfies both the horizontality condition
	\begin{equation}\label{eq:derivative condition}
		\dot{\alpha}_{122} = \dot{\alpha}_1X_1(\alpha) + \dot{\alpha}_2X_2(\alpha)
		= \frac{1}{12}\alpha_2^2\dot{\alpha}_1 - (\frac{1}{12}\alpha_1\alpha_2-\frac{1}{2}\alpha_{12})\dot{\alpha}_2
	\end{equation}
	and the initial condition $\alpha_{122}(0)=\frac{2}{3}$. The initial condition is immediately verified, since $\alpha_{122}(0) = \frac{4}{6\cosh(0)} = \frac{2}{3}$.
	
	Since $\alpha$ and $\beta$ agree in the first four coordinates, we get by Lemma~\ref{lemma:Engel explicit geodesic} that
	\begin{align*}
	\frac{1}{12}\alpha_2^2\dot{\alpha}_1 &= 
	\frac{1}{12}\Big(2\frac{\sinh(t)}{\cosh(t)}-t\Big)^2\Big(-\frac{2\sinh(t)}{\cosh(t)^2}\Big)
	\\&=-\frac{2\sinh(t)^3}{3\cosh(t)^4}+\frac{2t\sinh(t)^2}{3\cosh(t)^3}-\frac{t^2\sinh(t)}{6\cosh(t)^2}
	\\&=\frac{-4\sinh(t)^3+4t\sinh(t)^2\cosh(t)-t^2\sinh(t)\cosh(t)^2}{6\cosh(t)^4},
	\end{align*}
	\begin{align*}
	-\frac{1}{12}\alpha_1\alpha_2\dot{\alpha}_2 &=
	-\frac{1}{12}\frac{2}{\cosh(t)}\Big(\frac{2\sinh(t)}{\cosh(t)}-t\Big)\Big(1-\frac{2\sinh(t)^2}{\cosh(t)^2}\Big)
	\\&= -\frac{\sinh(t)}{3\cosh(t)^2}+\frac{2\sinh(t)^3}{3\cosh(t)^4}+\frac{t}{6\cosh(t)}-\frac{t\sinh(t)^2}{3\cosh(t)^3}
	\\&= \frac{-2\sinh(t)\cosh(t)^2+4\sinh(t)^3+t\cosh(t)^3-2t\sinh(t)^2\cosh(t)}{6\cosh(t)^4},
	\end{align*}
	and
	\begin{align*}
	\frac{1}{2}\alpha_{12}\dot{\alpha}_2 &=
	\frac{1}{2}\frac{t}{\cosh(t)}\Big(1-\frac{2\sinh(t)^2}{\cosh(t)^2}\Big)
	\\&=\frac{t}{2\cosh(t)}-\frac{t\sinh(t)^2}{\cosh(t)^3}
	\\&=\frac{3t\cosh(t)^3-6t\sinh(t)^2\cosh(t)}{6\cosh(t)^4}.
	\end{align*}
	Summing up the above, we get
	\begin{align*}
	\frac{1}{12}\alpha_2^2\dot{\alpha}_1 - (\frac{1}{12}\alpha_1\alpha_2-\frac{1}{2}\alpha_{12})\dot{\alpha}_2
	&= \frac{2t}{3\cosh(t)^3} - \frac{(t^2+2)\sinh(t)}{6\cosh(t)^2}.
	\end{align*}
	On the other hand, by differentiating the given expression for $\alpha_{122}$, we also get
	\begin{align*}
	\frac{d}{dt}\Big(\frac{t^{2} + 4}{6\cosh(t)} + \frac{t \sinh(t)}{3\cosh(t)^{2}}\Big)
	&=\frac{2t}{3\cosh(t)^3} - \frac{(t^2+2)\sinh(t)}{6\cosh(t)^2},
	\end{align*}
	so the horizontality condition \eqref{eq:derivative condition} is satisfied.
\end{proof}

\begin{proposition}\label{prop:different asymptote lines geodesic}
	Let $L_\pm(t)=\exp(-(tX_2\pm\frac{2}{3}tX_{1122}))$.
	Then
	\[ \sup_{t\in\R_+}d(\alpha(t),L_+(t)) < \infty\quad\text{and}\quad\sup_{t\in\R_-}d(\alpha(t),L_-(t))<\infty. \]
\end{proposition}
\begin{proof}
	As in Proposition~\ref{prop:Engel geodesic asymptotes}, we compute the distances $d(\alpha(t),L_\pm(t))$ directly by the considering the products $L_\pm(t)^{-1}\alpha(t)$. Since the lines $L_+$ and $L_-$ only differ by the sign of $\frac{2}{3}tX_{1122}$, we will combine the computations. That is, we will consider the product
	\[ z(t) := \exp(tX_2 \pm \tfrac{2}{3}tX_{1122})\alpha(t). \]
	
	The group law in the first four coordinates is exactly the group law of the Engel group \eqref{Engel:group law}, so the first four components $z_1,z_2,z_{12},z_{112}$ are bounded by Corollary~\ref{cor:Engel geodesic asymptotic properties}. It remains to consider the components $z_{122}$ and $z_{1122}$.
	
	By the group law \eqref{R2S4:group law}, we have
	\begin{align*}
	z_{122}(t) &= \alpha_{122}(t) 
	- \frac{1}{2}t\alpha_{12}(t)
	- \frac{1}{12}t\alpha_1(t)\alpha_2(t)
	+ \frac{1}{12}t^2\alpha_1(t)
	\quad\text{and}\\
	z_{1122}(t) &= \alpha_{1122}(t)
	\pm \frac{2}{3}t
	- \frac{1}{2}t\alpha_{112}(t)
	+ \frac{1}{12}t\alpha_1(t)\alpha_{12}(t)
	- \frac{1}{24}t^2\alpha_1(t)^2.
	\end{align*}
	By the explicit expressions given in Lemma~\ref{lemma:Engel explicit geodesic}, we see that the components $\alpha_1=\beta_1$ and $\alpha_{12}=\beta_{12}$ are both exponentially asymptotically vanishing, i.e., for any polynomial $P:\R\to\R$, we have
	\begin{equation*}
		\lim\limits_{t\to\pm\infty}P(t)\alpha_1(t) = \lim\limits_{t\to\pm\infty}P(t)\alpha_{12}(t) = 0.
	\end{equation*}
	Therefore there exists a constant $C>0$ such that
	\begin{equation}\label{lift asymptotes:final coordinate bounds}
	\abs{z_{122}(t)} \leq \abs{\alpha_{122}(t)} + C\quad\text{and}\quad
	\abs{z_{1122}(t)} \leq \abs{\alpha_{1122}(t)\pm \frac{2}{3}t- \frac{1}{2}t\alpha_{112}(t)} + C.
	\end{equation}
	By the explicit form in Lemma~\ref{lemma:R2S3:missing coordinate}, we see that $\alpha_{122}$ is bounded, so the same is true for $z_{122}$. For $z_{1122}$, we will consider the term
	\begin{align*}
	w(t)&:=\alpha_{1122}(t) \pm \frac{2}{3}t- \frac{1}{2}t\alpha_{112}(t)
	\end{align*}
	separately for $t>0$ and $t<0$.
	
	Instead of explicitly computing $\alpha_{1122}$, we will consider the derivative $\dot{w}$. Since $\alpha$ is a horizontal curve, from the explicit form \eqref{R2S4:horizontal frame} of the left-invariant frame, we get the identity
	\[ \dot{\alpha}_{1122} = (\frac{1}{12}\alpha_{12}\alpha_2-\frac{1}{2}\alpha_{122})\dot{\alpha}_1 + (\frac{1}{12}\alpha_1\alpha_{12} + \frac{1}{2}\alpha_{112})\dot{\alpha}_2. \]
	By the explicit expressions given in Lemmas \ref{lemma:Engel explicit geodesic} and \ref{lemma:R2S3:missing coordinate}, we see that as $t\to\pm\infty$ the terms $\alpha_1,\alpha_{12},\alpha_{122},\dot{\alpha}_{112}$ and $\dot{\alpha}_2+1$, are all exponentially vanishing. It follows that
	\begin{equation}\label{lift asymptotes:final coordinate derivative}
		\dot{w}(t) = -\alpha_{112}(t) \pm \frac{2}{3} + \epsilon(t),
	\end{equation}
	where $\epsilon:\R\to\R_+$ is some smooth function such that $\epsilon(t)=O(e^{-\abs{t}})$ as $t\to\pm\infty$.
	
	Finally, we observe that as $t\to\infty$, $\alpha_{112}(t)-\frac{2}{3}=O(e^{-t})$, and as $t\to -\infty$, $\alpha_{112}(t)+\frac{2}{3}=O(e^t)$. Therefore from \eqref{lift asymptotes:final coordinate derivative} we conclude that as $t\to\infty$ we have
	\[ \abs{\alpha_{1122}(t)+\frac{2}{3}t-\frac{1}{2}t\alpha_{112}(t)}
	\leq \int_{0}^{t}\abs{-\alpha_{112}(s) + \frac{2}{3} + \epsilon(s)}\,ds = O(e^{-t}). \]
	It follows from \eqref{lift asymptotes:final coordinate bounds} that also the final coordinate of $L_+(t)^{-1}\alpha(t)$ is bounded on $\R_+$. Thus the product $L_+(t)^{-1}\alpha(t)$ is bounded on $\R_+$.
	
	Similarly for $t\to-\infty$ we conclude that
	\[ \abs{\alpha_{1122}(t)-\frac{2}{3}t-\frac{1}{2}t\alpha_{112}(t)} = O(e^t), \]
	from which it follows that the product $L_-(t)^{-1}\alpha(t)$ is bounded on $\R_-$, proving the claim.
\end{proof}

\begin{corollary}\label{cor:geodesic not in a line nbhd}
	Let $L:\R\to G$ be any line. Then $\dhaus(\alpha(\R),L(\R))=\infty$.
\end{corollary}
\begin{proof}
	The corollary follows from combining Lemma~\ref{lemma:lines at a finite distance} and Proposition~\ref{prop:different asymptote lines geodesic}. 
	Suppose there existed a line $L\subset G$ such that $\dhaus(\alpha(\R),L(\R))<\infty$. Then also $\dhaus(\horproj\circ\alpha(\R),\horproj\circ L(\R))\leq M$, so from the explicit form of the horizontal components of $\alpha$ given in Lemma~\ref{lemma:Engel explicit geodesic}, we see that $\horproj\circ L$ must be parallel to the $x_2$-axis. 
	
	Up to reparametrizing $L$ we can then assume that $\horproj\circ L(t)=(C,-t)$ for some $C\in\R$. In particular, we have
	\[ d(\alpha(t),L(s))\geq d(\horproj\circ\alpha(t),\horproj\circ L(s))\geq \abs{t-s}-2. \]
	Then by Lemma~\ref{lemma:finite hausdorff distance}, since $\dhaus(\alpha(\R),L(\R))<\infty$, we have that also $\sup_{t}d(\alpha(t),L(t))<\infty$. In particular $\dhaus(\alpha(\R_+),L(\R_+))<\infty$ and $\dhaus(\alpha(\R_-),L(\R_-))<\infty$.
	
	Let $L_\pm(t)=  \exp(t Y_\pm)$ be the lines of Proposition~\ref{prop:different asymptote lines geodesic}. Proposition~\ref{prop:different asymptote lines geodesic} and the triangle inequality for the Hausdorff distance imply that
	\[ \dhaus(L(\R_+),L_+(\R_+)) \leq \dhaus(L(\R_+),\alpha(\R_+)) + \dhaus(\alpha(\R_+),L_-(\R_+)) < \infty \]
	and similarly that $\dhaus(L(\R_-),L_2(\R_-))<\infty$.
	By applying Lemma~\ref{lemma:lines at a finite distance} to both halves of the line $L $, we get the existence of constants $c_\pm>0$  such that 
	$$X=c_- \Ad_{g^{-1}} Y_- =c_+ \Ad_{g^{-1}} Y_+,$$
	where $X$ and $g$ are such that $L(t)= g\exp(tX)$.
	This implies that $Y_+$ and $Y_-$ are linearly dependent, which is a contradiction.
\end{proof}

Corollary~\ref{cor:geodesic not in a line nbhd} shows that $\alpha:\R\to G$ is a geodesic that is not in a finite neighborhood of any line, showing that the claim of Theorem~\ref{thm:geodesic blowdowns} cannot hold without considering the projection $\horproj:G\to G/\brkt{G}{G}$. Still Conjecture~\ref{conj:geodesic blowdowns} may be true.

\bibliography{general_bibliography} 
\bibliographystyle{amsalpha}
\end{document}